\documentclass[12pt,a4paper]{amsart}
\usepackage{geometry}
\usepackage{amsmath,amssymb,amsfonts,amscd,amsthm,wasysym,color,enumitem,indentfirst,graphicx,booktabs,mathrsfs}
\usepackage[bookmarksnumbered,colorlinks,linktocpage,plainpages]{hyperref}
\hypersetup{pdfstartview={FitH}}
\usepackage{mathtools}
\usepackage{extarrows}
\usepackage{tikz-cd}

\numberwithin{equation}{section}
\newtheorem{theorem}{Theorem}[section]
\newtheorem{corollary}[theorem]{Corollary}
\newtheorem{lemma}[theorem]{Lemma}
\newtheorem{proposition}[theorem]{Proposition}

\theoremstyle{definition}
\newtheorem{definition}[theorem]{Definition}

\newtheorem{example}[theorem]{Example}

\newtheorem{remark}[theorem]{Remark}

\usepackage{tikz}

\allowdisplaybreaks[4]

\makeatletter
\@namedef{subjclassname@2020}{\textup{2020} Mathematics Subject Classification}
\makeatother

\date{}

\newcommand{\AAAAA}{{\mathbf{A}^\prime}}
\begin{document}
\title[A Koszul complex in quaternionic analysis and its applications]{A Koszul complex in quaternionic analysis and its applications}
\date{\today}

\author{Yong Li}
\address[Yong Li]{School of Mathematics and Statistics, Anhui Normal University, Wuhu 241002, Anhui, People's Republic of China}
\email{leeey@ahnu.edu.cn}

\author{Yuchen Zhang*}
\address[Yuchen Zhang]{College of Mathematics and Systems Science, Shandong University of Science and Technology, Qingdao 266590, People's Republic of China}
\email{yuchen95@amss.ac.cn}

\thanks{
$\ast$ Corresponding author.
\\
This work is supported by the National Natural Science Foundation of China (Grant No. 12501100), the Anhui Provincial Natural Science Foundation (Grant No. 2508085QA031) and
the University Natural Science Research Project of Anhui Province (Grant No. 2022AH050175).}


\begin{abstract}
 Let $n\geqslant 1, \Omega\subset\mathbb{H}^n $ be a domain. We construct a Koszul-type complex for the ideal sheaf
$\mathcal{I}_X^{(k)}$ of $k$-regular functions vanishing on
$X=\{(q_0, q_1, \cdots, q_{n-1})\in \Omega: q_0=0\}$ in several quaternionic variables:
\[
0\to \mathcal{R}^{(k+2)}\xrightarrow{\widetilde{\mathscr{L}}^{(k)}}
\mathcal{R}^{(k+1)}\oplus\mathcal{R}^{(k+1)}\xrightarrow{\mathscr{L}^{(k)}}
\mathcal{I}_X^{(k)}\to 0,
\]
where $k\geqslant 0$,  $\mathcal{R}^{(k)}$ is the sheaf of $k$-regular functions on $\Omega$, $\widetilde{\mathscr{L}}^{(k)}\!=(-L_1^{(k+2)}\!,L_0^{(k+2)})^{T}$,
$\mathscr{L}^{(k)}\!=(L_0^{(k+1)}\!,L_1^{(k+1)})$, and $L_0^{(k)},L_1^{(k)}$ are
multiplication-like operators on $k$-regular functions.
This gives the quaternionic analogue of the classical Koszul complex.  And
we present the long exact sequence in cohomology for the case $\Omega\cap\{q_0=0\}=\emptyset$ with explicit differential connecting maps, by applying the Cauchy–Fueter complex and cohomological methods.

 As an application, in the special case $n=1, k=0$, the operator pair $(L_0^{(1)}, L_1^{(1)})$ is shown to be surjective if and only if $H^3(\Omega, \mathbb{R})=0$. Furthermore, a cohomological vanishing criterion is given for $H^1(\Omega,\mathcal{I}_X^{(k)})$; under this criterion, every $k$-regular function on $\{q_0=0\}\cap\Omega$ extends to a $k$-regular function on $\Omega$.
\end{abstract}

\maketitle
\tableofcontents
\section{Introduction}

The quaternion algebra $\mathbb{H}$, with its familiar basis $\{1, \mathbf{i}, \mathbf{j}, \mathbf{k}\}$ and multiplication rules $\mathbf{i}^2=\mathbf{j}^2=\mathbf{k}^2=\mathbf{ijk}=-1$, was introduced by Hamilton \cite{Hamilton66}. Quaternionic analysis, broadly understood as a generalization of complex analysis to the non-commutative setting, has since evolved into a rich subject encompassing several distinct but interrelated frameworks.

The most direct quaternionic counterpart of holomorphic function theory is the theory of \textbf{quaternionic regular functions}, initiated by Fueter \cite{Fueter35}. A differentiable function $f$ is called (left) regular if it is annihilated by the Cauchy--Fueter operator
\begin{equation*}
D = \frac{\partial}{\partial x_0}+\mathbf{i}\frac{\partial}{\partial x_1}+\mathbf{j}\frac{\partial}{\partial x_2}+\mathbf{k}\frac{\partial}{\partial x_3}.
\end{equation*}
Fueter showed that these functions enjoy many properties parallel to those of holomorphic functions --- notably Laurent expansions and a Cauchy integral formula \cite{Fueter35,Fueter36} --- and they have found numerous applications in theoretical physics \cite{Adler86,Lanczos29}. This theory has been extensively developed over the ensuing decades \cite{Sudbery79, CSSS04} and has been subsumed into the broader programme of Clifford analysis \cite{BDS82,Gilbert91,LY24,CSSS04}.

A natural higher-spin generalization arises from the \textbf{$k$-Cauchy--Fueter operator} $\mathscr{D}_0^{(k)}$ (see Definition~\ref{def:kCF}), which originates in physics as the elliptic incarnation of the spin-$k/2$ massless field equations on Minkowski spacetime \cite{CMW16,EPW80,PR84,PR86}. Functions annihilated by $\mathscr{D}_0^{(k)}$ are called \textbf{$k$-regular functions}; they reduce to Fueter's regular functions when $k=1$, while the case $k=0$ corresponds to harmonic functions. Thus the $k$-regular hierarchy interpolates between harmonic and regular functions, offering a flexible scale of regularity.

The extension of the $k$-Cauchy--Fueter operator to several quaternionic variables was carried out by Wang \cite{Wang10}, where the resulting $k$-regular functions simultaneously generalize Fueter's one-variable theory and the several-variable regular functions introduced earlier by Pertici \cite{Pertici88}. In this multivariable setting, there is a one-to-one correspondence between $1$-regular functions and quaternionic regular functions \cite{Wang10}, and between $0$-regular functions and quaternionic pluriharmonic functions \cite{Alesker03,AV10,WW17,Sroka20}. Kang and Wang \cite{KW13} subsequently established a Taylor expansion for $k$-regular functions via the Penrose integral formula. The associated \textbf{Cauchy--Fueter complex} --- the quaternionic analogue of the Dolbeault complex --- has been studied by several authors \cite{ALPS97,ABLSS99,Wang10}, and a comprehensive treatment of $k$-Cauchy--Fueter operators and $k$-regular functions can be found in the series of works \cite{Wang10,Wang15,Wang19,Wang25,RZ23}.

The Koszul complex is a very important tool in algebra \cite{Buch64,BR64,GH78,CSSS04}. As an application in the theory of several complex variables, the Koszul complex provides a resolution of the ideal sheaf $\mathcal{I}_X$, where $X$ is an analytic subset of $\mathbb{C}^n$ and $\mathcal{I}_X$ is defined as the sheaf of  holomorphic functions vanishing on $X$. This resolution further shows that $\mathcal{I}_X$ is a coherent analytic sheaf, and the complex has many applications and generalizations in several complex variables \cite{GH78,Ji12,JY26}.

We illustrate with the example
\[
X=\{(z_0,z_1,\dots,z_{n-1})\in \mathbb{C}^n : z_0=z_1=0\}.
\]
Indeed, the Koszul complex tells us that we have the following exact sequence of sheaves:
\[
0 \longrightarrow \mathcal{O} \xrightarrow{(-z_1,z_0)^T} \mathcal{O}^2 \xrightarrow{(z_0,z_1)} \mathcal{I}_X \longrightarrow 0,
\]
where $\mathcal{O}$ is the sheaf of  holomorphic functions (see for example \cite[P.698]{GH78}).

We wish to obtain analogous results for quaternionic regular functions or $k$-regular functions. However, for a quaternionic regular function $f(q), q\in \mathbb{H}$, $qf$ is no longer quaternionic regular, so we cannot directly generalise the above result. However, Sudbery's result \cite{Sudbery79} tells us that in this case $qf$ is harmonic. Note that quaternionic regular functions $f$ are in one-to-one correspondence with $1$-regular functions, while harmonic functions are exactly $0$-regular functions. Inspired by this, Li and Zhang constructed a pair of multiplication operators $L_0,L_1$ in the one-variable $\mathbb{H}$ case, and proved that for a $k$-regular function $f$, both $L_0f$ and $L_1f$ are $(k-1)$-regular functions \cite{LZ26}.

In this paper, we further extend this pair of multiplication operators $L_0^{(k)}, L_1^{(k)}$ (see definition in \ref{def:L_0L_1 on k-CF}) to $\mathbb{H}^n$, and thereby obtain results analogous to the several complex variable case.

\begin{theorem}[Koszul resolution of $\mathcal{I}_X^{(k)}$] \label{thm:main}
Let $\Omega\subset\mathbb{H}^n$ be a domain and set
\[
X:=\{(q_0, q_1, \cdots, q_{n-1})\in \Omega: q_0=0\}.
\]
For every integer $k\geqslant 0$, the sequence
\begin{equation}\label{eq:Koszul complex}
0\longrightarrow \mathcal{R}^{(k+2)} \xrightarrow{\widetilde{\mathscr{L}}^{(k)}} \mathcal{R}^{(k+1)}\oplus \mathcal{R}^{(k+1)} \xrightarrow{\mathscr{L}^{(k)}} \mathcal{I}_X^{(k)} \longrightarrow 0
\end{equation}
is an exact sequence of sheaves on $\Omega$, where $\widetilde{\mathscr{L}}^{(k)}$ and $\mathscr{L}^{(k)}$ are defined by
\[
\widetilde{\mathscr{L}}^{(k)}:=(-L_1^{(k+2)}, L_0^{(k+2)})^T,\qquad
\mathscr{L}^{(k)}:=(L_0^{(k+1)}, L_1^{(k+1)}),
\]
respectively.
\end{theorem}
Here $\mathcal{R}^{(k)}$ denotes the sheaf of  $k$-regular functions on $\Omega$, and $\mathcal{I}_X^{(k)}$ is defined as the sheaf of  $k$-regular functions vanishing on $X$.

The proof of this theorem is divided into two parts, for $x\in X$ and $x\notin X$. The proof for $x\notin X$ is similar to the treatment of  Theorem 1.2 in \cite{LZ26}. For the case $x\in X$, we need a theorem analogous to Hefer's Lemma in several complex variables. In fact, using the Taylor expansion of $k$-regular functions, we prove:
\begin{theorem}
Let $0<\varepsilon\ll 1$, and $$B_0(\varepsilon)=\{(q_0,\cdots, q_{n-1})\in \mathbb{H}^n: \sqrt{\left|q_0\right|^2+\left|q_1\right|^2+\cdots+\left|q_{n-1}\right|^2}<\varepsilon \}.$$
If $h\in \mathcal{R}^{(k)}(B_0(\varepsilon))$ and $h|_{\{q_0=0\}}=0$, then there exist $f,g\in \mathcal{R}^{(k+1)}(B_0(\varepsilon^2))$ satisfying $L_0^{(k+1)}f + L_1^{(k+1)}g = h$  on $ B_0(\varepsilon^2)$.

\end{theorem}

It is easy to see that this is a generalisation of the following Hefer's Lemma.
\begin{theorem}[Hefer's lemma{\cite[P.244]{Kra92}}]
  Let $\Omega \subseteq \mathbb{C}^2$ be pseudoconvex. Assume that $0\in \Omega$. If $f:\Omega\to \mathbb{C} $ is a holomorphic function that satisfies $f(0)=0$, then there are holomorphic functions $f_1, f_2$ on $\Omega$ such that $f=z_1f_1+z_2f_2$ on $\Omega$.
\end{theorem}

Li and Zhang \cite{LZ26} first obtained a quaternionic generalization of Hefer's lemma:

\begin{theorem}[{\cite[Theorem~1.2]{LZ26}}]
Let $\Omega\subset\mathbb{R}^4$ be a domain with $H^3(\Omega,\mathbb{R})=0$
and $0\notin\Omega$. Then for any harmonic function $h\in C^\infty(\Omega,\mathbb{C})$
there exists a pair of $1$-regular functions $f,g$ such that
$h = L_0^{(1)}f + L_1^{(1)}g$.
\end{theorem}

As a first application of Theorem~\ref{thm:main}, we strengthen the above result
to an \emph{if and only if} statement and also treat the case $0\in\Omega$.

\begin{theorem}\label{thm:2}
Let $\Omega\subset\mathbb{R}^4$ be a domain.
\begin{itemize}
\item If $0\notin\Omega$, then $H^3(\Omega,\mathbb{R})=0$ if and only if for every harmonic function $h\in C^\infty(\Omega,\mathbb{C})$ there exist $1$-regular functions $f,g$ such that $h=L_0^{(1)}f+L_1^{(1)}g$.
\item If $0\in\Omega$, then $H^3(\Omega,\mathbb{R})=0$ if and only if for every harmonic function $h\in C^\infty(\Omega,\mathbb{C})$ with $h(0)=0$ there exist $1$-regular functions $f,g$ such that $h=L_0^{(1)}f+L_1^{(1)}g$.
\end{itemize}
\end{theorem}

By standard sheaf theory, the short exact sequence in
Theorem~\ref{thm:main} induces a long exact sequence of cohomology
groups, which plays a central role in the proof of
Theorem~\ref{thm:2}.  In general, for every $k\geqslant 0$ we have
\begin{equation*}
\begin{aligned}
0\rightarrow & H^0(\Omega,\mathcal{R}^{(k+2)})\stackrel{\widetilde{\mathcal{L}}_0^{(k)}}{\longrightarrow}H^0\big(\Omega,\mathcal{R}^{(k+1)}\oplus \mathcal{R}^{(k+1)}\big)\stackrel{\mathcal{L}_0^{(k)}}{\longrightarrow} H^0(\Omega,\mathcal{R}^{(k)})\\
\stackrel{s_0^{(k)}}{\longrightarrow} & H^1(\Omega,\mathcal{R}^{(k+2)}) \stackrel{\widetilde{\mathcal{L}}_1^{(k)}}{\longrightarrow}H^1\big(\Omega,\mathcal{R}^{(k+1)}\oplus \mathcal{R}^{(k+1)}\big)\stackrel{\mathcal{L}_1^{(k)}}{\longrightarrow} H^1(\Omega,\mathcal{R}^{(k)}) \\
\stackrel{s_1^{(k)}}{\longrightarrow} & H^2(\Omega,\mathcal{R}^{(k+2)}) \stackrel{\widetilde{\mathcal{L}}_2^{(k)}}{\longrightarrow}H^2\big(\Omega,\mathcal{R}^{(k+1)}\oplus \mathcal{R}^{(k+1)}\big)\stackrel{\mathcal{L}_2^{(k)}}{\longrightarrow} H^2(\Omega,\mathcal{R}^{(k)}) \\
\stackrel{s_2^{(k)}}{\longrightarrow} & \cdots
\end{aligned}
\end{equation*}

In practice, one works not with the definitions of sheaf cohomology but with the isomorphic cohomology of differential forms associated to $k$-regular functions, which is far more amenable to computation.  It is therefore natural to ask whether the maps in the above long exact sequence admit explicit expressions in terms of differential operators.  When $X=\Omega\cap \{q_0=0\}=\emptyset$, the answer is affirmative.  In this case the ideal sheaf $\mathcal{I}_X^{(k)}$ coincides with $\mathcal{R}^{(k)}$, whose differential-form resolution is provided by the $k$-Cauchy--Fueter complex. Using this complex, all connecting maps can be identified with explicit differential operators (see Definition~\ref{def: L and tilde L} and Proposition~\ref{prop: connecting homomorphism s}). We thus obtain the following.
\begin{theorem}[Long exact sequence in cohomology]\label{thm:long exact sequence}
Let $\Omega\subset\mathbb{H}^n$ be a domain with $\Omega\cap \{q_0=0\}=\emptyset$. Then for every $k\geqslant 0$ there is a long exact sequence
\begin{equation*}
\begin{aligned}
0&\rightarrow H^0(\Omega,\mathcal{R}^{(k+2)})\stackrel{\widetilde{\mathcal{L}}_0^{(k)}}{\longrightarrow}H^0\big(\Omega,\mathcal{R}^{(k+1)}\oplus \mathcal{R}^{(k+1)}\big)\stackrel{\mathcal{L}_0^{(k)}}{\longrightarrow} H^0(\Omega,\mathcal{R}^{(k)})\stackrel{s_0^{(k)}}{\longrightarrow} H^1(\Omega,\mathcal{R}^{(k+2)})\rightarrow  \\
\cdots &\rightarrow H^k\big(\Omega,\mathcal{R}^{(k+1)}\oplus \mathcal{R}^{(k+1)}\big)\stackrel{\mathcal{L}_k^{(k)}}{\longrightarrow} H^k(\Omega,\mathcal{R}^{(k)})\stackrel{s_k^{(k)}}{\longrightarrow}H^{k+1}(\Omega,\mathcal{R}^{(k+2)})\\
&\stackrel{\widetilde{\mathcal{L}}_{k+1}^{(k)}}{\longrightarrow} H^{k+1}\big(\Omega,\mathcal{R}^{(k+1)}\oplus \mathcal{R}^{(k+1)}\big).
\end{aligned}
\end{equation*}
\end{theorem}

We note that the sequence terminates at $H^{k+1}\big(\Omega,\mathcal{R}^{(k+1)}\oplus\mathcal{R}^{(k+1)}\big)$. This is because the $k$-Cauchy--Fueter complex involves a second-order differential operator in degree $k$, which forces the subsequent map $\mathcal{L}_{k+1}^{(k)}$ to be a first order differential operator; exactness beyond this point cannot be verified within the present framework.

When $\Omega\cap \{q_0=0\}\neq\emptyset$, an explicit description of the connecting maps is no longer available. Nevertheless, by means of \v{C}ech cohomology we can still obtain a useful criterion for the vanishing of $H^1(\Omega,\mathcal{I}_X^{(k)})$, to which we now turn.

\begin{theorem}[Cohomology vanishing criteria]\label{thm:main2}
Let $n,k\geqslant 1$ and let $\Omega\subset\mathbb{R}^{4n}$ be a domain.
\begin{enumerate}
\item If $X=\Omega\cap\{q_0=0\}\neq\emptyset$, then $H^1(\Omega,\mathcal{I}_X^{(k)})=0$ if and only if $\widetilde{\mathcal{L}}_1^{(k)}$ is surjective and $\widetilde{\mathcal{L}}_2^{(k)}$ is injective.
\item If $\Omega\cap\{q_0=0\}=\emptyset$, then $H^k(\Omega,\mathcal{R}^{(k)})=0$ if and only if $\widetilde{\mathcal{L}}_k^{(k)}$ is surjective and $\widetilde{\mathcal{L}}_{k+1}^{(k)}$ is injective.
\end{enumerate}
\end{theorem}

From the standard short exact sequence of quotient sheaves
\begin{equation*}
    0\longrightarrow \mathcal{I}_X^{(k)}    \longrightarrow \mathcal{R}^{(k)}     \longrightarrow \mathcal{R}^{(k)}_X \longrightarrow 0,
\end{equation*}
where $\mathcal{R}^{(k)}_X := \mathcal{R}^{(k)}/\mathcal{I}_X^{(k)}$, we obtain a long exact sequence in cohomology.  The connecting map
$H^0(\Omega,\mathcal{R}^{(k)}_X) \to H^1(\Omega,\mathcal{I}_X^{(k)})$ shows that whenever $H^1(\Omega,\mathcal{I}_X^{(k)})=0$, the restriction map $\mathcal{R}^{(k)}(\Omega) \to \mathcal{R}^{(k)}_X(\Omega)$ is surjective --- precisely the condition that every $k$-regular function on $X$ extends to a $k$-regular function on $\Omega$. Theorem~\ref{thm:main2}(i) supplies exactly such a vanishing criterion, and together they yield:

\begin{corollary}[Extension criterion via $\widetilde{\mathcal{L}}$-operators]\label{cor:1}
Let $k\geqslant 1$ and $n\geqslant 2$, and let $\Omega\subset\mathbb{R}^{4n}$ be a domain with $X=\Omega\cap\{q_0=0\}\neq\emptyset$. If $\widetilde{\mathcal{L}}_1^{(k)}$ is surjective and $\widetilde{\mathcal{L}}_2^{(k)}$ is injective, then every $k$-regular function on $X$ admits a $k$-regular extension to $\Omega$.
\end{corollary}

A particularly practical sufficient condition is obtained when the relevant cohomology groups of $\mathcal{R}^{(k+1)}$ and
$\mathcal{R}^{(k+2)}$ vanish.

\begin{corollary}[Extension under cohomology vanishing]\label{cor:2}
Let $k\geqslant 0$ and let $\Omega\subset\mathbb{R}^{4n}$ be a domain with $X=\Omega\cap\{q_0=0\}\neq\emptyset$. If
\[
H^1(\Omega,\mathcal{R}^{(k+1)})=0 \quad\text{and}\quad H^2(\Omega,\mathcal{R}^{(k+2)})=0,
\]
then every $k$-regular function on $X$ extends to a $k$-regular function on $\Omega$. In particular, this condition is automatically satisfied when $\Omega$ is convex.
\end{corollary}

The organization of this paper is as follows. In Section~\ref{section:Preliminaries} we recall necessary background on $k$-Cauchy--Fueter operators and $k$-regular functions. Section~\ref{section:L0L1} extends the multiplication-like operators $L_0^{(k)},L_1^{(k)}$ to $\mathbb{H}^n$ and establishes their basic properties. In Section~\ref{section:Koszul complex} we prove Theorem~\ref{thm:main}, the Koszul resolution of $\mathcal{I}_X^{(k)}$. Section~\ref{section:long exact sequence} constructs the long exact sequence of differential-form cohomology (Theorem~\ref{thm:long exact sequence}) with explicit differential operators as connecting maps. Finally, Section~\ref{section:Application} presents the applications: surjectivity of $(L_0^{(1)},L_1^{(1)})$, vanishing criteria and extension theorems.

\section{Preliminaries}\label{section:Preliminaries}

\subsection{Cauchy-Fueter operator}

  Let $\mathbb{H}$ be the quaternion space with the imaginary units $\mathbf{i},\mathbf{j},\mathbf{k}$  satisfying the relations:
$$\mathbf{i}^2=\mathbf{j}^2=\mathbf{k}^2=\mathbf{ijk}=-1.$$

Let  $\mathbb H^n$ be   the set of several quaternionic variables. We denote
$$x=\left(q_0,q_1,\ldots,q_{n-1}\right)\in \mathbb H^n,$$
where
$$q_t=x_{4t}+\mathbf{i}x_{4t+1}+\mathbf{j}x_{4t+2}+\mathbf{k}x_{4t+3}\in\mathbb H$$
for any $t=0,1,\ldots,n-1 $ with all $x_j$ being real.

 The Cauchy-Fueter operator is defined as follows
\begin{align*}
  &D: C^\infty(\mathbb{R}^{4n},\mathbb{H})\longrightarrow C^\infty(\mathbb{R}^{4n},\mathbb{H}^n),\\
    &(Df)_t: = \left(\frac{\partial}{\partial x_{4t}}+\mathbf{i}\frac{\partial}{\partial x_{4t+1}}+\mathbf{j}\frac{\partial}{\partial x_{4t+2}}+\mathbf{k}\frac{\partial}{\partial x_{4t+3}}\right)f,\ \ \  t=0,1,\ldots,n-1.
\end{align*}
Let $\Omega\subset\mathbb{R}^{4n}$ be an open subset and $f\in C^1(\Omega,\mathbb{H})$. If $Df=0$ on  $\Omega$, we call $f$ \textbf{ quaternionic regular} on $\Omega$. Let  $\mathcal{R}$ denote the sheaf of quaternionic regular functions on a domain $\Omega\subset \mathbb{R}^{4n}$.

\subsection{$k$-Cauchy-Fueter operators}

  As many properties of the $k$-Cauchy-Fueter equation and $k$-regular functions have already been investigated by Wang \cite{Wang10,Wang15,Wang19,Wang25,KW13}, here we briefly review the definition of the $k$-Cauchy-Fueter operator, along with basic notations and concepts such as the $k$-Cauchy-Fueter complex, $k$-regular functions.

We embed the quaternion ring $\mathbb{H}$ into the ring of complex $2\times 2$ matrices via
\begin{equation}\label{eq:z AAprime}
\begin{aligned}
\tau(q_t):&=\left(\begin{array}{cc}
    x_{4t}+ x_{4t+1}i & -x_{4t+2}- x_{4t+3}i  \\
     x_{4t+2}- x_{4t+3}i& x_{4t}- x_{4t+1}i
\end{array}\right)\\
&=
\left(\begin{array}{cc}
    \overline{z^{(2t)0^\prime}} & \overline{z^{(2t)1^\prime}}  \\
    \overline{z^{(2t+1)0^\prime}} & \overline{z^{(2t+1)1^\prime}}
\end{array}\right),\qquad t=0,\ldots,n-1.
\end{aligned}
\end{equation}
This yields  operators  $\nabla_{AA^\prime}$,   defined by
\begin{equation*}
  \begin{aligned}
  \left(\begin{array}{cc}
  \nabla_{(2t)0^\prime}   & \nabla_{(2t)1^\prime}\\
  \nabla_{(2t+1)0^\prime} & \nabla_{(2t+1)1^\prime}
  \end{array}\right):&=2\left(\begin{array}{cc}
    \frac{\partial}{{\partial z^{(2t)0^\prime}}} & \frac{\partial}{{\partial z^{(2t)1^\prime}}}  \\
    \frac{\partial}{{\partial z^{(2t+1)0^\prime}}} & \frac{\partial}{{\partial z^{(2t+1)1^\prime}}}
\end{array}\right)\\
&=
\left(
\begin{array}{cc}
\dfrac{\partial}{\partial  x_{4t} }+i\dfrac{\partial}{\partial x_{4t+1}}   & -\dfrac{\partial}{\partial  x_{4t+2} }-i\dfrac{\partial}{\partial x_{4t+3}} \\
 \dfrac{\partial}{\partial  x_{4t+2} }-i\dfrac{\partial}{\partial x_{4t+3}} &\  \ \ \dfrac{\partial}{\partial  x_{4t} }-i\dfrac{\partial}{\partial x_{4t+1}}
\end{array}
\right).
\end{aligned}
\end{equation*}

We can write smooth quaternions-valued functions $u,f$ as $u=u_0+\mathbf{j}u_1$ and $f=f_0+\mathbf{j}f_1$ for complex-valued functions $u_0,u_1,f_0$ and $f_1$. Then the Cauchy-Fueter equation
$$Du=f$$
is equivalent to the system:
\begin{equation*}
    \left(\begin{array}{ll}
  \nabla_{(2t)0^\prime}   & \nabla_{(2t)1^\prime}\\
  \nabla_{(2t+1)0^\prime} & \nabla_{(2t+1)1^\prime}
  \end{array}\right)
  \left(\begin{array}{l}
       u_0 \\
       u_1
  \end{array}\right)=
  \left(\begin{array}{l}
       f_0 \\
       f_1
  \end{array}\right), \qquad t=0,\ldots, n-1.
\end{equation*}

By raising  indices, we introduce derivatives
$$\nabla^{A^\prime}_{A}:=\sum_{B^\prime=0, 1} \nabla_{AB^\prime} \ \epsilon^{B^\prime A^\prime},$$
  where
 \begin{equation*}
\epsilon=(\epsilon^{B^\prime A^\prime})=\left(\begin{array}{cc}
  0 & -1 \\
  1 & 0
\end{array}\right)=\tau(\mathbf{j}).
\end{equation*}
 More precisely, we have
\begin{equation}\label{preliminary:nabla AA prime}
  \left(\begin{array}{ll}
  \nabla_{2t}^{0^\prime}   & \nabla_{2t}^{1^\prime}\\
  \nabla_{2t+1}^{0^\prime} & \nabla_{2t+1}^{1^\prime}
  \end{array}\right)=
\left(
\begin{array}{ll}
-\dfrac{\partial}{\partial x_{4t+2}}-i\dfrac{\partial}{\partial x_{4t+3}} & -\dfrac{\partial}{\partial x_{4t}}-i\dfrac{\partial}{\partial x_{4t+1}} \\
\ \ \ \dfrac{\partial}{\partial x_{4t}}-i\dfrac{\partial}{\partial x_{4t+1}}   & -\dfrac{\partial}{\partial x_{4t+2}}+i\dfrac{\partial}{\partial x_{4t+3}}
\end{array}
\right).
\end{equation}

Before introducing the \(k\)-Cauchy--Fueter operator, we first establish some necessary notations.
Fix a basis \(\{\omega^0,\omega^1,\ldots,\omega^{2n-1}\}\) of \(\mathbb{C}^{2n}\). An element of \(\Lambda^p\mathbb{C}^{2n}\) is denoted by
\[
f=\sum_{0\leqslant i_1<\cdots < i_p\leqslant 2n-1} f_{i_1\cdots i_p}\,\omega^{i_1}\wedge\cdots\wedge \omega^{i_p}.
\]

An element of $\mathbb{C}^2$ is denoted by
\[
(f_{A'}) \quad \text{or} \quad (f^{A'}), \qquad A'=0',1',
\]
where $f_{A'},f^{A'}\in\mathbb{C}$. The symmetric power $\odot^k\mathbb{C}^2$ is a subspace of $\otimes^k\mathbb{C}^2$ whose elements can be expressed as
\[
(f_{A_0'A_1'\cdots A_{k-1}'}) \quad \text{or} \quad (f^{A_0'A_1'\cdots A_{k-1}'}), \qquad A_0',A_1',\dots,A_{k-1}'=0',1',
\]
where $f_{A_0'A_1'\cdots A_{k-1}'},f^{A_0'A_1'\cdots A_{k-1}'}\in\mathbb{C}$ are invariant under permutations of their indices, i.e.,
\[
f_{A_0'A_1'\cdots A_{k-1}'}=f_{A_{\sigma(0)}'A_{\sigma(1)}'\cdots A_{\sigma(k-1)}'}\quad \text{ and } \quad  f^{A_0'A_1'\cdots A_{k-1}'}=f^{A_{\sigma(0)}'A_{\sigma(1)}'\cdots A_{\sigma(k-1)}'}
\]
for any permutation $\sigma\in S_k$, the symmetric group on $k$ letters. For the special case $k=0$, we set $\odot^0\mathbb{C}^2\cong\mathbb{C}$.

Let $k\geqslant0$. Consider the complex vector spaces
\[
V^{(k)}_j=
\begin{cases}
\odot^{k-j}\mathbb{C}^2\otimes \Lambda^{j}\mathbb{C}^{2n},
& j=0,\ldots,k,\\[6pt]
\odot^{j-k-1}\mathbb{C}^2\otimes \Lambda^{j+1}\mathbb{C}^{2n},
& j=k+1,\ldots,2n-1.
\end{cases}
\]
Note that \(V_k^{(k)}\cong \Lambda^k\mathbb{C}^{2n}\) and \(V_{k+1}^{(k)}\cong \Lambda^{k+2}\mathbb{C}^{2n}\). Let \(j\leqslant k-1\). An element of \(\odot^{k-j}\mathbb{C}^2\otimes\Lambda^{j}\mathbb{C}^{2n}\) is denoted by
\[
(f_{ A_0^\prime\cdots A_{k-j-1}^\prime}),\qquad A_0^\prime,\ldots,A_{k-j-1}^\prime=0^\prime,1^\prime,
\]
where each \(f_{A_0^\prime\cdots A_{k-j-1}^\prime}\in\Lambda^j\mathbb{C}^{2n}\) is invariant under permutations of the indices
\[A_0^\prime,\dots,A_{k-j-1}^\prime.\]
For fixed \(j>k+1\), an element of \(\odot^{j-k-1}\mathbb{C}^2\otimes\Lambda^{j+1}\mathbb{C}^{2n}\) is denoted by
\[
(f^{ A_0^\prime\cdots A_{j-k-2}^\prime}),\qquad A_0^\prime,\ldots,A_{j-k-2}^\prime=0^\prime,1^\prime,
\]
where each \(f^{A_0^\prime\cdots A_{j-k-2}^\prime}\in\Lambda^{j+1}\mathbb{C}^{2n}\) is invariant under permutations of the indices
\[A_0^\prime,\dots,A_{j-k-2}^\prime.\]

We also use the notation $\AAAAA$ to denote the indices $A_0^\prime A_1^\prime \cdots A_{p}^\prime$ for some nonnegative integer $p$.

Within the theory of quaternionic analysis, two fundamental operators $d^{0'}$ and $d^{1'}$ frequently appear in relevant research works (cf. \cite{WW17,Wang25}). We present their definitions as follows.
\begin{equation}\label{eq:def of d0d1}
\begin{aligned}
&d^{j'}\colon C^\infty(\Omega,\Lambda^p\mathbb{C}^{2n})\longrightarrow C^\infty(\Omega,\Lambda^{p+1}\mathbb{C}^{2n}),\qquad j'=0',1',\\
&d^{j'}f:=\sum_{A=0}^{2n-1}\sum_{0\leqslant A_1<\cdots < A_p\leqslant 2n-1} \nabla_{A}^{j'} f_{A_1\cdots A_p}\,\omega^A\wedge\omega^{A_1}\wedge\cdots\wedge \omega^{A_p}.
\end{aligned}
\end{equation}
Let $\Lambda^\ast\mathbb{C}^{2n}$ denote the exterior algebra. Then the operators $d^{0'}$ and $d^{1'}$ are naturally defined on $C^\infty(\Omega,\Lambda^\ast\mathbb{C}^{2n})$ and satisfy the relations below, see for example \cite{WW17}:
\begin{equation}\label{eq:prop of d0 and d1}
d^{0'}d^{0'}=0,\qquad d^{1'}d^{1'}=0, \qquad d^{0'}d^{1'}+d^{1'}d^{0'}=0.
\end{equation}

Now we introduce the $k$-Cauchy-Fueter operators
\begin{equation*}
\begin{aligned}
  \mathscr{D}_0^{(k)}: & C^\infty(\mathbb{R}^{4n},\odot^k\mathbb{C}^2)\longrightarrow C^\infty(\mathbb{R}^{4n},\Lambda^1\mathbb{C}^{2n}\otimes\odot^{k-1}\mathbb{C}^2),
\end{aligned}
\end{equation*}
defined by
 \begin{equation}\label{def:kCF}
   \left(\mathscr{D}_0^{(k)}f\right)_{A_0^\prime\cdots A_{k-2}^\prime} :=
  \begin{cases} d^{0^\prime}f_{0^\prime}+d^{1^\prime}f_{1^\prime},  & \qquad k=1,\\  \\ d^{0^\prime}f_{0^\prime A_0^\prime\cdots A_{k-2}^\prime}+ d^{1^\prime}f_{1^\prime A_0^\prime\cdots A_{k-2}^\prime}, & \qquad   k\geqslant2.
  \end{cases}
\end{equation}
Let $k,n\geqslant 1$, let $\Omega\subset\mathbb{R}^{4n}$ be an open set, and let $f\in C^1\big(\Omega,\odot^k\mathbb{C}^2\big)$. We say that $f$ is \textbf{$k$-regular} on $\Omega$ if $\mathscr{D}^{(k)}_0f=0$ holds everywhere on $\Omega$. If $d^{0'}d^{1'}f=0$ on $\Omega$, then $f$ is said to be \textbf{$0$-regular}, or quaternionic pluriharmonic, on $\Omega$ \cite{Alesker03,AV10,Sroka20,WW17}. In addition, vectors in $\odot^k\mathbb{C}^2$ are defined as $k$-regular functions at a single point. Denote by $\mathcal{R}^{(k)}$  the sheaf of $k$-regular functions on $\Omega$.

Various concrete examples of k-regular functions have been constructed in the literature; we refer the reader to \cite{Wang10,Wang15,LZ26} for detailed illustrations. Some typical examples are listed as follows.

\begin{example}
$(1)$  Let $\Omega\subset\mathbb{R}^{4n}$ be an open subset. Then function $f=(f_{0'},f_{1'})\in C^1(\Omega,\mathbb{C}^2)$ is $1$-regular on $\Omega$ if and only if $\mathbf{j}(f_{0'}+\mathbf{j}f_{1'})$ is a quaternionic regular function on $\Omega$.

$(2)$ We provide an explicit characterization of $2$-regular functions. If $f\in C^\infty(\Omega,\odot^2\mathbb{C}^2)$ is $2$-regular, then $f$ satisfies the following system of equations for each $t=0,1,\dots,n-1$:
\begin{equation*}
\begin{aligned}
&\nabla_{2t}^{0^\prime}f_{0^\prime0^\prime}+\nabla_{2t}^{1^\prime}f_{1^\prime0^\prime}=0,\qquad \nabla_{2t+1}^{0^\prime}f_{0^\prime0^\prime}+\nabla_{2t+1}^{1^\prime}f_{1^\prime0^\prime}=0,\\
&\nabla_{2t}^{0^\prime}f_{0^\prime1^\prime}+\nabla_{2t}^{1^\prime}f_{1^\prime1^\prime}=0,\qquad \nabla_{2t+1}^{0^\prime}f_{0^\prime1^\prime}+\nabla_{2t+1}^{1^\prime}f_{1^\prime1^\prime}=0, \qquad f_{0^\prime1^\prime}=f_{1^\prime0^\prime}.
\end{aligned}
\end{equation*}

$(3)$ Consider a quaternionic regular function $F=F_1\mathbf{i}+F_2\mathbf{j}+F_3\mathbf{k}$ with real-valued components $F_1,F_2,F_3$. We define a smooth function $f\in C^\infty(\Omega,\odot^2\mathbb{C}^2)$ by
\[f_{0^\prime0^\prime}=F_2-iF_3, \qquad   f_{0^\prime1^\prime}=f_{1^\prime0^\prime}=-iF_1,   \qquad f_{1^\prime1^\prime}=F_2+iF_3.\]
Then, $f$ is 2-regular.

$(4)$ If $g\in C^\infty(\Omega,\odot^k\mathbb{C}^2)$ is $k$-regular on $\Omega$, then every component $g_{A_0^\prime A_1^\prime\cdots A_{k-1}^\prime}$ of $g$ is $0$-regular on $\Omega$. When $n=1$, the function $f$ is $0$-regular on $\Omega$ if and only if $f$ is harmonic on $\Omega$.
\end{example}

Let \(E_j^{(k)}\) denote the trivial bundle over \(\Omega\) with fiber \(V_j^{(k)}\),
and let \(\mathcal{V}_j^{(k)}\) be the sheaf of smooth sections of \(E_j^{(k)}\).
Then the sheaf \(\mathcal{R}^{(k)}\) admits an acyclic resolution \cite{Wang10}:
\begin{equation}\label{eq:resolution_of_k_CF}
0\rightarrow \mathcal{R}^{(k)}
\xrightarrow{\,i\,}
\mathcal{V}_0^{(k)}
\xrightarrow{\mathscr{D}_0^{(k)}}
\mathcal{V}_1^{(k)}
\xrightarrow{\mathscr{D}_1^{(k)}}
\mathcal{V}_2^{(k)}
\xrightarrow{\mathscr{D}_2^{(k)}}
\cdots.
\end{equation}
Here \(i\) is the natural embedding, and the operators
\[
\mathscr{D}_j^{(k)}:
C^\infty(\mathbb{R}^{4n},V_j^{(k)})
\longrightarrow
C^\infty(\mathbb{R}^{4n},V_{j+1}^{(k)})
\]
are defined by
\begin{equation}\label{eq:def_of_Dj}
(\mathscr{D}_j^{(k)}f)_{\AAAAA}
=
\begin{cases}
d^{0^\prime}f_{0^\prime \AAAAA}
+ d^{1^\prime}f_{1^\prime \AAAAA},
& 0\leqslant j\leqslant k-2,
\\[4pt]
d^{0^\prime}f_{0^\prime} + d^{1^\prime}f_{1^\prime},
& j = k-1,
\\[4pt]
d^{0^\prime}d^{1^\prime}f,
& j = k,
\\[4pt]
d^{A_0^\prime}f,
& j = k+1,
\\[4pt]
d^{(A_{0}^\prime}f^{A_{1}^\prime \cdots A_{p(k,j)}^\prime)},
& j\geqslant k+2,
\end{cases}
\end{equation}
where $\AAAAA$ denotes the indices $A_0^\prime\cdots A_{p(k,j)}^\prime$ and
\begin{equation}\label{eq:def:p kj}
p(k,j)=
\begin{cases}
k-j-1, & j=0,\dots,k-1,
\\[4pt]
j-k-2, & j=k+2,\dots,2n.
\end{cases}
\end{equation}

Here $d^{(A_0^\prime}f^{A_1^\prime \cdots A_{p(k,j)}^\prime)}$ denotes the symmetrisation of $d^{A_0^\prime}f^{A_1^\prime \cdots A_{p(k,j)}^\prime}$ with respect to the indices $A_{0}'\cdots A_{p(k,j)}'$.
Let $f\in \otimes^p\mathbb{C}^2\otimes \Lambda^q\mathbb{C}^{2n}$. The symmetrisation over the indices of $f$ is defined as follows:
 \[f^{(A_0^\prime \ldots A_{p-1}^\prime)}:=\frac{1}{p!}\sum_{\sigma\in S_p}f^{A_{\sigma(0)}^\prime \ldots A_{\sigma(p-1)}^\prime}.\]

Of course, the direct definition of $\mathscr{D}_j^{(k)}$ is computationally intractable. We shall provide an alternative formulation of $\mathscr{D}_j^{(k)}$ via a suitable isomorphism.

Note that there exists an isomorphism between the following vector spaces:
\[
\eta \colon \odot^{p}\mathbb{C}^2\otimes\Lambda^\ast\mathbb{C}^{2n} \longrightarrow \mathbb{C}^{p+1}\otimes \Lambda^\ast\mathbb{C}^{2n}.
\]
This isomorphism is defined as
\begin{equation}\label{eq:def:eta}
(\eta(f))_{s} := f_{A_0'\cdots A_{p-1}'}, \quad \text{where } A_0 + \cdots + A_{p-1} = s,\quad s=0,1,\dots ,p.
\end{equation}
Here $(\eta(f))_{s}$ denotes a function taking values in differential forms.

 Under this isomorphism, for each $j<k$, the operator $\mathscr{D}_j^{(k)}$ can be written as
 \begin{equation}\label{eq:Djk under eta j small}
     \eta\bigl(\mathscr{D}_j^{(k)}f\bigr)_s=d^{0'}\eta(f)_s+d^{1'}\eta(f)_{s+1}.
 \end{equation}

For the cases where $j\geqslant k+2$, we need the following lemma.

\begin{lemma}\label{lem:sym}\cite{Wang17}
If \((F_{A_0' \dots A_p'}) \in \otimes^{p+1} \mathbb{C}^2 \otimes \Lambda^q \mathbb{C}^{2n}\) is symmetric in \(A_1', \dots, A_p'\), then we have
\[
F_{(A_0' \dots A_p')} = \frac{1}{p+1} \sum_{s=0}^{p} F_{A_s' A_0' \dots \widehat{A_s'} \dots A_p'},
\]
where the first term in the sum corresponds to \(F_{A_0' \dots A_p'}\).
\end{lemma}

For each $j\geqslant k+2$, it follows from Lemma \ref{lem:sym} that
\begin{equation}\label{eq:Djk under eta j large}
\eta\big(\mathscr{D}_j^{(k)}f\big)_s
=\frac{p(k,j)+1-s}{p(k,j)+1}\,d^{0'}\eta(f)_s
+\frac{s}{p(k,j)+1}\,d^{1'}\eta(f)_{s-1}.
\end{equation}
In addition, we adopt the convention that $\eta(f)_{-1}:=0$.

Equations \eqref{eq:Djk under eta j small} and \eqref{eq:Djk under eta j large} facilitate the calculations in the proof of our main results.

At the end of this section, we note that
\[
H^j(\Omega,\mathcal{R}^{(k)})\cong\frac{\ker\mathscr{D}_j^{(k)}}{\operatorname{Im}\mathscr{D}_{j-1}^{(k)}},\qquad j\geqslant 1.
\]
This identity holds since the resolution \eqref{eq:resolution_of_k_CF} is acyclic.

\subsection{Ideal sheaves in quaternionic analysis and extension problem}

Extending holomorphic functions from a subvariety or submanifold to the entire space is a central theme in several complex variables and complex geometry. Ideal sheaves serve as an indispensable tool in these studies. This section explores the extension of \(k\)-regular functions from
\[
\{\,q_0=0\,\}\cap\Omega
\]
to the entire domain \(\Omega\), as well as the role that ideal sheaves play in this process.

Let $\Omega\subset \mathbb{R}^{4n}$ be a domain. Consider the function spaces of $k$-regular functions on $\Omega$ that vanish on
\[
X:=\{\,q_0=0\,\}\cap\Omega.
\]
We define the presheaf $\mathcal{I}_X^{(k)}$ by
\begin{equation*}
\mathcal{I}_X^{(k)}(U):=\bigl\{f\in \mathcal{R}^{(k)}(U) : f(x)=0 \text{ for } x\in X\bigr\}.
\end{equation*}
It is clear that if $X=\emptyset$, then $\mathcal{I}^{(k)}_X$ is exactly the sheaf $\mathcal{R}^{(k)}$.

\begin{proposition}
    The presheaf $\mathcal{I}_X^{(k)}$ is a sheaf.
\end{proposition}
\begin{proof}
    By the definition of a sheaf, it suffices to show that if $\Omega_\alpha$ is an open cover of $\Omega$, then $\mathcal{I}_X^{(k)}$ satisfies
    \begin{enumerate}
        \item If $f,g\in \mathcal{I}_X^{(k)}(\Omega)$ and $f=g$ on every $\Omega_\alpha$, then $f=g$ on $\Omega$;
        \item If $f_\alpha\in \mathcal{I}_X^{(k)}(\Omega_\alpha)$ and if for $\Omega_\alpha\cap\Omega_\beta\neq \emptyset$ we have
\[f_\alpha=f_\beta\quad \text{ on } \quad \Omega_\alpha\cap\Omega_\beta\]
for all $\alpha,\beta$, then there exists $F\in \mathcal{I}_X^{(k)}$ such that $F|_{\Omega_\alpha}=f_\alpha$.
    \end{enumerate}
The statement $(1)$ is obvious. Consider statement $(2)$. It is clear that there exists a function $F\in C^\infty(\Omega,\odot^k\mathbb{C}^2)$ such that $F|_{\Omega_\alpha}=f_\alpha$. It remains to show that $F\in \mathcal{I}_X^{(k)}(\Omega)$. For any $x\in\Omega$, there exists some $\Omega_\alpha$ containing $x$ such that
\[
(\mathscr{D}_0^{(k)}F)(x)=(\mathscr{D}_0^{(k)}f_\alpha)(x)=0.
\]
If in addition $x\in X$, then $F(x)=f_\alpha(x)=0$. Therefore, $F\in \mathcal{I}_X^{(k)}(\Omega)$. This completes the proof.
\end{proof}

Notice that the sheaf $\mathcal{I}_X^{(k)}$ is a subsheaf of $\mathcal{R}^{(k)}$. We then obtain a short exact sequence of sheaves
\begin{equation}\label{eq:short exact seq of IkX}
0\longrightarrow \mathcal{I}_X^{(k)}\longrightarrow \mathcal{R}^{(k)}\longrightarrow \mathcal{R}^{(k)}_X\longrightarrow 0.
\end{equation}
Here $\mathcal{R}_X^{(k)}$ denotes the quotient sheaf $\mathcal{R}^{(k)}/\mathcal{I}_X^{(k)}$. From the definition of the quotient sheaf, we recall that
\[
\mathcal{R}^{(k)}_{X,x}\cong \mathcal{R}^{(k)}_x/\mathcal{I}_{X,x}^{(k)}.
\]
For any open set $U\subset \Omega$,
\[
\mathcal{R}^{(k)}_{X}(U)=
\begin{cases}
\mathcal{R}^{(k)}(U\cap X), & \text{if } U\cap X \neq \emptyset,\\
0, & \text{if } U\cap X = \emptyset,
\end{cases}
\]
where $\mathcal{R}^{(k)}(U\cap X)$ represents the space of $k$-regular functions defined on $X\cap U\subset\mathbb{R}^{4n-4}$.
 For any open sets $V\subset U\subset \Omega$, the map $r^U_V$ is the corresponding restriction map.

It is clear that \(H^0(\Omega,\mathcal{R}^{(k)}_X)\cong \mathcal{R}^{(k)}(\Omega\cap X)\) coincides with the set of \(k\)-regular functions on \(\Omega\cap X\) for \(n>1\), and is isomorphic to the complex vector space \(\odot^k\mathbb{C}^2\) when \(n=1\). Note that the short exact sequence \eqref{eq:short exact seq of IkX} induces the following long exact sequence:
\[
0\longrightarrow H^0(\Omega,\mathcal{I}_X^{(k)})\longrightarrow H^0(\Omega,\mathcal{R}^{(k)})\stackrel{i'}\longrightarrow H^0(\Omega,\mathcal{R}^{(k)}_X)\stackrel{s'}\longrightarrow H^1(\Omega,\mathcal{I}_X^{(k)})\longrightarrow\cdots.
\]
If a \(k\)-regular function \(f\) on \(\Omega\cap X\) admits a \(k\)-regular extension to the entire space \(\Omega\), then there exists \(F\in H^0(\Omega,\mathcal{R}^{(k)})\) such that \(i'(F)=f\). This yields the following proposition.

\begin{proposition}\label{prop:extension}
    Let $\Omega\subset\mathbb{H}^n$ be a domain such that $X=\{q_0=0\}\cap \Omega\neq \emptyset$. If a $k$-regular function $f$ on $\Omega\cap X$ satisfies $s'(f)=0$, then there exists a function $F$ that is $k$-regular on $\Omega$ and $F|_X=f$. In particular, if $H^1(\Omega,\mathcal{I}_X^{(k)})=0$, then every $k$-regular function $f$ on $\Omega\cap X$ admits a $k$-regular extension to the entire domain $\Omega$.
\end{proposition}

\section{The multiplication-like operators $L_A$}\label{section:L0L1}

Sudbery \cite{Sudbery79} observed that if $f$ is quaternionic regular, then $qf$ is harmonic. In previous work \cite{LZ26}, Li and Zhang constructed a pair of multiplication-type operators $L_0$ and $L_1$, which map $k$-regular functions on $\mathbb{R}^4$ to $(k-1)$-regular functions. In this section, we generalize these operators to $k$-regular functions over $\mathbb{R}^{4n}$.

Applying the embedding $\tau$ defined in \eqref{eq:z AAprime} to $-q_t\mathbf{j}$, we obtain
\begin{equation*}
\begin{aligned}
\tau(-q_t\mathbf{j})
&= \tau\bigl(x_{4t+2} + x_{4t+3}\mathbf{i} - x_{4t}\mathbf{j} - x_{4t+1}\mathbf{k}\bigr)\\
&= \begin{pmatrix}
x_{4t+2} + x_{4t+3}i & x_{4t} + x_{4t+1}i \\
-x_{4t} + x_{4t+1}i & x_{4t+2} - x_{4t+3}i
\end{pmatrix}\\
&= \begin{pmatrix}
z^{(2t+1)0'} & z^{(2t+1)1'} \\
-z^{(2t)0'} & -z^{(2t)1'}
\end{pmatrix}.
\end{aligned}
\end{equation*}

We denote the matrix $\tau(-q_t\mathbf{j})$ by
\begin{equation}\label{eq:def:zAAprime}
\begin{pmatrix}
z_{2t}^{0'} & z_{2t}^{1'}\\[4pt]
z_{2t+1}^{0'} & z_{2t+1}^{1'}
\end{pmatrix}.
\end{equation}

\begin{definition}\label{def:L_0L_1 on k-CF}
Let $k\geqslant1$ and let $\Omega$ be a domain in $\mathbb{R}^{4n}$.
For $A=0,1,\dots,2n-1$, the operators $L^{(k)}_A$ are defined on $C^\infty(\Omega,\odot^k\mathbb{C}^2)$ by
\begin{equation*}
\begin{aligned}
&L_A^{(k)}:C^\infty(\Omega,\odot^k\mathbb{C}^{2}\otimes\Lambda^\ast\mathbb{C}^{2n})\longrightarrow C^\infty(\Omega,\odot^{k-1}\mathbb{C}^2\otimes\Lambda^\ast\mathbb{C}^{2n}),\\
&(L^{(k)}_Af)_{A'_0\cdots A'_{k-2}}
:= z_A^{0'} f_{0'A'_0\cdots A'_{k-2}}
+ z_A^{1'} f_{1'A'_0\cdots A'_{k-2}}.
\end{aligned}
\end{equation*}
\end{definition}

Under the isomorphism $\eta$, we obtain the following computationally convenient formula:
\begin{equation}\label{eq:LA eta}
\eta\big(L_A^{(k)}f\big)_s=z_A^{0'}\eta(f)_s+z_A^{1'}\eta(f)_{s+1}
\end{equation}
which holds for all $s=0,1,\dots,k-1$.

Let $k\geqslant 1$. We define the operator
\[
D^{(k)}:C^\infty\big(\Omega,\odot^{k}\mathbb{C}^{2}\otimes \Lambda^\ast\mathbb{C}^{2n}\big)\longrightarrow C^\infty\big(\Omega,\odot^{k-1}\mathbb{C}^{2}\otimes \Lambda^\ast\mathbb{C}^{2n}\big)
\]
via the formula
\begin{equation}\label{eq:def Dk}
\eta(D^{(k)}f)_s:=d^{0'}\eta(f)_s+d^{1'}\eta(f)_{s+1},\quad s=0,\ldots,k-1.
\end{equation}
We note that for all $j<k$, the operators $\mathscr{D}_j^{(k)}$ share an identical structural form with $D^{(k-j)}$. To be more specific, under the isomorphism $\eta$, both $\mathscr{D}_j^{(k)}$ and $D^{(k-j)}$ can be expressed as $d^{0'}\eta(\cdot)_s+d^{1'}\eta(\cdot)_{s+1}$.

We generalize Theorem 1.1 and Proposition 3.2 in \cite{LZ26}.

\begin{proposition}\label{prop:Lj and k CF}
The following identities hold.
\begin{enumerate}
    \item The identity
    \[
    D^{(k-1)} L_A^{(k)}= L_A^{(k-1)} D^{(k)}
    \]
    holds for all $k\geqslant2, A=0,1,\ldots,2n-1$.
    \item Let $k\geqslant1$. Suppose $f\in C^\infty(\Omega,\odot^k\mathbb{C}^{2})$ is $k$-regular on a domain $\Omega\subset\mathbb{R}^{4n}$. Then $L_A^{(k)}f$ is $(k-1)$-regular on $\Omega$.
\end{enumerate}
\end{proposition}

 Before presenting the proof, we first state a useful lemma for subsequent computations.
\begin{lemma}\label{lem:computation z}\cite[Lemma 3.3]{Wang17}
\begin{enumerate}
\item $\nabla_{AA^\prime} z^{BB^\prime}=2\delta_A^{B}\delta_{A^\prime}^{B^\prime}$.
\item We define $\epsilon^{BA}$ by
\begin{equation}\label{eq:def:epsilon BA}
\epsilon^{BA}:=\begin{cases} -1, & \qquad B=2t,A=2t+1, t=0,\ldots,n-1,\\
1, & \qquad B=2t+1,A=2t, t=0,\ldots,n-1,\\
0, & \qquad \text{ otherwise.}
\end{cases}
\end{equation}
It then follows that
\[\nabla_A^{A^\prime}z_{B}^{B^\prime}=-2\epsilon^{B^\prime A^\prime}\epsilon^{BA}.\]
\item Let $f\in C^1(\mathbb{R}^{4n},\Lambda^p\mathbb{C}^{2n})$. Then the equation
\[d^{j^\prime}(z_{A}^{A^\prime}f)=z_A^{A^\prime}d^{j^\prime}f-\sum_{B=0}^{2n-1}2\epsilon^{BA}\epsilon^{j^\prime A^\prime}\omega^B\wedge f\]
holds for all $j^\prime,A^\prime=0^\prime,1^\prime$ and $A=0,1,\ldots,2n-1$.
\end{enumerate}
\end{lemma}
\begin{proof}
The assertion $(1)$ corresponds exactly to Lemma $3.3$ in \cite{Wang17}. We split the proof of assertion $(2)$ into two cases. If $A\in \{2s,2s+1\}$, $B\in\{2t,2t+1\}$ and $s\neq t$, then from the definitions of $\nabla_A^{A^\prime}$ \eqref{preliminary:nabla AA prime} and $z_A^{A^\prime}$ \eqref{eq:def:zAAprime}, we have
\[
\nabla_A^{A^\prime}z_B^{B^\prime}=0.
\]
The case where $A\in \{2s,2s+1\}$ and $B\in\{2s,2s+1\}$ can be proved similarly to the proof of Lemma 3.3 in \cite{LZ26}.
Assertion $(3)$ follows from $(2)$ and the definition of $d^{j^\prime}$ (see equation \eqref{eq:def of d0d1}).
\end{proof}

 \emph{Proof of Proposition \ref{prop:Lj and k CF}.}
We first prove the assertion $(1)$. Based on the definition of $L_A^{(k)}$ (see Definition \ref{def:L_0L_1 on k-CF}), the definition of $\eta$ in \eqref{eq:def:eta}, and the definition of $D^{(k)}$ in \eqref{eq:def Dk}, the identity below holds for all integers $s = 0,1,\dots,k-2$:
\[
\begin{aligned}
\eta\bigl(D^{(k-1)}(L_A^{(k)}f)\bigr)_{s}
&= d^{0^\prime}\eta(L_A^{(k)}f)_{s}
+ d^{1^\prime}\eta(L_A^{(k)}f)_{s+1} \\
&= d^{0^\prime}\bigl(z_A^{0^\prime}\eta(f)_{s}
+ z_A^{1^\prime}\eta(f)_{s+1}\bigr) + d^{1^\prime}\bigl(z_A^{0^\prime}\eta(f)_{s+1}
+ z_A^{1^\prime}\eta(f)_{s+2}\bigr).
\end{aligned}
\]
Applying Lemma \ref{lem:computation z}, we further obtain
\begin{equation}\label{eq:prop Lj and k CF: k geq 2:main}
\begin{aligned}
 \eta\bigl(D^{(k-1)}(L_A^{(k)}f)\bigr)_{s} &=z_A^{0^\prime}\left(d^{0^\prime}\eta(f)_{s}+d^{1^\prime}\eta(f)_{s+1}\right)+z_A^{1^\prime}\left(d^{0^\prime}\eta(f)_{s+1}+d^{1^\prime}\eta(f)_{s+2}\right)\\
 &\quad -2\sum_{B=0}^{2n-1}\left(\epsilon^{0^\prime 0^\prime}\epsilon^{AB}\omega^B\wedge \eta(f)_{s}+\epsilon^{1^\prime 0^\prime}\epsilon^{AB}\omega^B\wedge \eta(f)_{s+1}\right.\\
 &\quad +\left.\epsilon^{0^\prime 1^\prime}\epsilon^{AB}\omega^B\wedge \eta(f)_{s+1}+\epsilon^{1^\prime 1^\prime}\epsilon^{AB}\omega^B\wedge \eta(f)_{s+2}\right).
 \end{aligned}
\end{equation}

Recall the formula \eqref{eq:Djk under eta j small} and note that
\[
\epsilon=\begin{pmatrix}
      0 & -1\\
    1  & 0
\end{pmatrix}.
\]
Combining the foregoing identities with \eqref{eq:prop Lj and k CF: k geq 2:main}, we arrive at
\[
\begin{aligned}
\eta\bigl(D^{(k-1)}L_A^{(k)}f\bigr)_s&=z_A^{0^\prime}\left(d^{0^\prime}\eta(f)_{s}+d^{1^\prime}\eta(f)_{s+1}\right)+z_A^{1^\prime}\left(d^{0^\prime}\eta(f)_{s+1}+d^{1^\prime}\eta(f)_{s+2}\right)\\
&=z_A^{0'}\eta\bigl(D^{(k)}f\bigr)_s+z_A^{1'}\eta\bigl(D^{(k)}f\bigr)_{s+1}=\eta\bigl(L_A^{(k-1)}D^{(k)}f\bigr)_s.
\end{aligned}
\]

Now we prove assertion $(2)$. The case $k\geqslant 2$ follows directly from assertion $(1)$. It suffices to prove the case $k=1$. Let $\Omega\subset\mathbb{R}^{4n}$ be a domain, and let $f\in C^1(\Omega,\mathbb{C}^2)$ be a $1$-regular function on $\Omega$.
We show that
\[
d^{0^\prime}d^{1^\prime}\bigl(z_A^{0^\prime}f_{0^\prime}+z_A^{1^\prime}f_{1^\prime}\bigr)=0.
\]
By Lemma \ref{lem:computation z},
\begin{align*}
d^{0^\prime}d^{1^\prime}\bigl(z_A^{0^\prime}f_{0^\prime}+z_A^{1^\prime}f_{1^\prime}\bigr)
&= z_A^{0^\prime}d^{0^\prime}d^{1^\prime}f_{0^\prime}
+ z_A^{1^\prime}d^{0^\prime}d^{1^\prime}f_{1^\prime} \\
&\quad+ 2\sum_{B=0}^{2n-1}\Bigl(
\epsilon^{0^\prime 1^\prime}\epsilon^{AB}\omega^B\wedge d^{0^\prime}f_{0^\prime}
- \epsilon^{1^\prime 0^\prime}\epsilon^{AB}\omega^B\wedge d^{1^\prime}f_{1^\prime}
\Bigr)\\
&= z_A^{0^\prime}d^{0^\prime}d^{1^\prime}f_{0^\prime}
+ z_A^{1^\prime}d^{0^\prime}d^{1^\prime}f_{1^\prime}- 2\sum_{B=0}^{2n-1}\epsilon^{AB}\omega^B\wedge
\bigl(d^{0^\prime}f_{0^\prime}+d^{1^\prime}f_{1^\prime}\bigr)\\
&=z_A^{0^\prime}d^{0^\prime}d^{1^\prime}f_{0^\prime}
+ z_A^{1^\prime}d^{0^\prime}d^{1^\prime}f_{1^\prime}.
\end{align*}
On the other hand, recall the properties of $d^{0'}$ and $d^{1'}$ stated in \eqref{eq:prop of d0 and d1}. We deduce that
\begin{equation*}
    \begin{aligned}
     &d^{0'}d^{1'}f_{0'}=-d^{1'}(d^{0'}f_{0'}+d^{1'}f_{1'})=0,\\
     &d^{0'}d^{1'}f_{1'}=d^{0'}(d^{0'}f_{0'}+d^{1'}f_{1'})=0.
    \end{aligned}
\end{equation*}
The above equations imply that
\[d^{0'}d^{1'}\bigl(z_A^{0'}f_{0'} + z_A^{1'}f_{1'}\bigr)=0,\]
which finishes the proof. \qed

\begin{proposition}\label{prop: Lj Li commutative}
    Let $k\geqslant 2$. For any fixed indices $A,B=0,1,\ldots,2n-1$, we have
    \[L_A^{(k-1)}L_B^{(k)}=L_B^{(k-1)}L_A^{(k)}.\]
\end{proposition}
\begin{proof}
The proof of this proposition is analogous to that of Proposition 3.2 in \cite{LZ26}. Let \(f\in C^\infty(\Omega,\odot^2\mathbb{C}^2)\). Then
\begin{equation}\label{eq:prop L0L1:proof of 1:1}
\begin{aligned}
L^{(1)}_AL^{(2)}_Bf
&= z_A^{0^\prime}(L_B^{(2)}f)_{0^\prime}+z_A^{1^\prime}(L_B^{(2)}f)_{1^\prime}\\[4pt]
&= z_A^{0^\prime}\bigl(z_B^{0^\prime} f_{0^\prime0^\prime}+z_B^{1^\prime} f_{1^\prime0^\prime}\bigr)
 + z_A^{1^\prime}\bigl(z_B^{0^\prime} f_{0^\prime1^\prime}+z_B^{1^\prime} f_{1^\prime1^\prime}\bigr)\\[4pt]
&= z_B^{0^\prime}\bigl(z_A^{0^\prime}f_{0^\prime0^\prime}+z_A^{1^\prime}f_{1^\prime0^\prime}\bigr)
 + z_{B}^{1^\prime}\bigl(z_A^{0^\prime}f_{0^\prime1^\prime}+z_A^{1^\prime}f_{1^\prime1^\prime}\bigr)\\[4pt]
&= z_B^{0^\prime}(L_A^{(2)}f)_{0^\prime}+z_B^{1^\prime}(L_A^{(2)}f)_{1^\prime}
 = L_B^{(1)}L_A^{(2)}f.
\end{aligned}
\end{equation}
This proves the statement for the case \(k=2\).

The proof for the cases \(k\geqslant 3\) and \(f\in C^\infty(\Omega,\odot^k\mathbb{C}^2)\) follows by applying identity \eqref{eq:prop L0L1:proof of 1:1} to the term \(\eta\bigl(L^{(k-1)}_AL^{(k)}_B f\bigr)_{s}\). More precisely, we have
\begin{equation*}
\begin{aligned}
\eta\bigl(L^{(k-1)}_AL^{(k)}_B f\bigr)_{s}
&= z_A^{0^\prime}\bigl(z_B^{0^\prime} \eta(f)_{s}
+ z_B^{1^\prime} \eta(f)_{s+1}\bigr)+ z_A^{1^\prime}\bigl(z_B^{0^\prime} \eta(f)_{s+1}
+ z_B^{1^\prime} \eta(f)_{s+2}\bigr) \\
&= z_B^{0^\prime}\bigl(z_A^{0^\prime}\eta(f)_{s}
+ z_A^{1^\prime}\eta(f)_{s+1}\bigr) + z_B^{1^\prime}\bigl(z_A^{0^\prime}\eta(f)_{s+1}+z_A^{1^\prime}\eta(f)_{s+2}\bigr)\\
&= \eta\bigl(L^{(k-1)}_BL^{(k)}_Af\bigr)_{s}.
\end{aligned}
\end{equation*}
\end{proof}

\section{A Koszul resolution of $\mathcal{I}_X^{(k)}$}\label{section:Koszul complex}

In this section we prove Theorem~\ref{thm:main}.  We first define the sheaf morphisms \(\widetilde{\mathscr{L}}^{(k)}\) and \(\mathscr{L}^{(k)}\) in terms of the operators \(L_A^{(k)}\), and show that they form a complex resolving \(\mathcal{I}_X^{(k)}\).  Verifying exactness reduces to a stalkwise computation: away from \(X\) we use the invertibility of the coefficient matrix, while on \(X\) we expand $k$-regular functions in the basis \(P^{\mathbf{m}}_{s,j}\) of \cite{KW13}.

We define operators $\mathscr{L}^{(k)}$ and $\widetilde{\mathscr{L}}^{(k)}$ as follows:
\begin{definition}
    Let $k\geqslant0$ and $\Omega\subset\mathbb{R}^{4n}$ be a domain. We define operators ${\mathscr{L}}^{(k)}$ and $\widetilde{\mathscr{L}}^{(k)}$ as follows:
    \[
    \begin{aligned}
&\mathscr{L}^{(k)}:\mathcal{R}^{(k+1)}\oplus\mathcal{R}^{(k+1)}\longrightarrow \mathcal{I}_X^{(k)}, \qquad  \left(\mathscr{L}^{(k)}\left(   f_0,  f_1 \right)^T\right)_x:=\left(L^{(k+1)}_0 f_0+L^{(k+1)}_1 f_1\right)_x,\\
    &\widetilde{\mathscr{L}}^{(k)}:\mathcal{R}^{(k+2)}\longrightarrow \mathcal{R}^{(k+1)}\oplus\mathcal{R}^{(k+1)}, \qquad \left(\widetilde{\mathscr{L}}^{(k)}g\right)_x:= \left( (-L^{(k+2)}_1 g)_x,   (L^{(k+2)}_0 g)_x\right)^T.
    \end{aligned}
    \]
    where $x\in \Omega$.
\end{definition}

We now verify exactness on stalks, splitting the proof into the two cases: \(x\notin X\) and \(x\in X\).

\subsection{Proof of Theorem \ref{thm:main}: Case $x\notin X$}

In this section, we suppose $x \notin X$. Under this assumption, the sheaf $\mathcal{I}_X^{(k)}$ coincides with $\mathcal{R}^{(k)}$. We aim to prove that the sequence
\[
0\longrightarrow \mathcal{R}^{(k+2)}_x\xrightarrow{\widetilde{\mathscr{L}}^{(k)}}\mathcal{R}^{(k+1)}_x\oplus\mathcal{R}^{(k+1)}_x\xrightarrow{\mathscr{L}^{(k)}}\mathcal{R}_{x}^{(k)}\longrightarrow 0
\]
is exact.

The proof is divided into several steps.

\textbf{Step 1: The injectivity of $\widetilde{\mathscr{L}}^{(k)}_x$.}

Let $U$ be an open neighborhood of $x$ such that $U\cap X=\emptyset$. Let $p\geqslant1$. We shall establish the following stronger assertion: if $h\in C^\infty\big(U,\odot^{p}\mathbb{C}^2\otimes \Lambda^\ast\mathbb{C}^{2n}\big)$ obeys the identities $L_0^{(p)}h = L_1^{(p)}h = 0$, then $h$ vanishes identically on $U$.

Assume that $\eta(L_0^{(p)}h)_{s} = \eta(L_1^{(p)}h)_{s} = 0$ holds for each $s=0,1,\ldots p-1$. It follows from equation \eqref{eq:LA eta} that
\[
\begin{pmatrix}
-z_1^{0'} & -z_1^{1'}\\[4pt]
z_0^{0'}  &  z_0^{1'}
\end{pmatrix}
\begin{pmatrix}
\eta(h)_{s}\\[4pt]
\eta(h)_{s+1}
\end{pmatrix}
= 0.
\]
Note that the determinant of the matrix
\[
\begin{pmatrix}
-z_1^{0'} & -z_1^{1'}\\[4pt]
z_0^{0'}  &  z_0^{1'}
\end{pmatrix}
\]
is equal to $z_0^{0'}z_1^{1'} - z_0^{1'}z_1^{0'} = |q_0|^2$. It follows that
\[
\eta(h)_{s} = \eta(h)_{s+1} = 0
\]
for all indices $s$ on $U$. Hence $h = 0$.

\textbf{Step 2: The exactness of $\mathcal{R}^{(k+2)}_x\xrightarrow{\widetilde{\mathscr{L}}^{(k)}}\mathcal{R}^{(k+1)}_x\oplus\mathcal{R}^{(k+1)}_x\xrightarrow{\mathscr{L}^{(k)}}\mathcal{R}_{x}^{(k)}$.}

Let $U$ be an open neighborhood of $x$ such that $U\cap X=\emptyset$. We now prove a stronger result. Suppose $f,g\in C^\infty\big(U,\odot^{k+1}\mathbb{C}^2\otimes \Lambda^\ast\mathbb{C}^{2n}\big)$ satisfy
\begin{equation}\label{eq:thm main: x notin X:mid:assume of f and g}
D^{(k+1)}f=D^{(k+1)}g=0, \qquad L_0^{(k+1)}f+L_1^{(k+1)}g=0.
\end{equation}
Clearly, every pair of $(k+1)$-regular functions $f,g$ with $\mathscr{L}^{(k)}(f,g)^T=0$ obeys the above system of equations. We aim to construct a function
\[h\in C^\infty\big(U,\odot^{k+2}\mathbb{C}^2\otimes \Lambda^\ast\mathbb{C}^{2n}\big)\]
satisfying
\begin{equation}\label{eq:thm main: x notin X:mid:goal h}
f=-L_1^{(k+2)}h,\qquad g=L_0^{(k+2)}h,\qquad D^{(k+2)}h=0.
\end{equation}
It is straightforward to verify that if $f,g\in \mathcal{R}^{(k+1)}(U)$, then the function $h$ constructed as above lies in $\mathcal{R}^{(k+2)}(U)$ and satisfies
\[\widetilde{\mathscr L}^{(k)}h=(f,g)^T.\]

Assume that $f,g$ satisfies equation \eqref{eq:thm main: x notin X:mid:assume of f and g}. Then $h$ satisfies the condition in \eqref{eq:thm main: x notin X:mid:goal h} if and only if equality
\begin{equation*}
\begin{pmatrix}
-z_1^{0'} & -z_1^{1'}\\
z_0^{0'} & z_0^{1'}
\end{pmatrix}
\begin{pmatrix}
\eta(h)_s\\
\eta(h)_{s+1}
\end{pmatrix}=
\begin{pmatrix}
\eta(f)_s\\
\eta(g)_s
\end{pmatrix},
\end{equation*}
holds for all $s=0,1,\dots,k+1$. Consequently, we define $\eta(h)$ by
\begin{equation*}
\begin{pmatrix}
\eta(h)_s\\
\eta(h)_{s+1}
\end{pmatrix}= \frac{1}{\vert q_0\vert^2}
\begin{pmatrix}
z_0^{1'} & z_1^{1'}\\
-z_0^{0'} & -z_1^{0'}
\end{pmatrix}
\begin{pmatrix}
\eta(f)_s\\
\eta(g)_s
\end{pmatrix}.
\end{equation*}
It should be noted that for each $s=1,\dots,k+1$, the equality
\begin{equation}\label{eq:thm main: x notin X:mid:1}
\eta(h)_s=\frac{1}{\vert q_0\vert^2}\big(z_0^{1'}\eta(f)_{s}+z_1^{1'}\eta(g)_s\big)
\end{equation}
holds, and we also have
\[
\eta(h)_s=\frac{1}{\vert q_0\vert^2}\big(-z_0^{0'}\eta(f)_{s-1}-z_1^{0'}\eta(g)_{s-1}\big).
\]
It is necessary to verify that
\begin{equation}\label{eq:thm main: x notin X:mid:2}
z_0^{1'}\eta(f)_{s}+z_1^{1'}\eta(g)_s+z_0^{0'}\eta(f)_{s-1}+z_1^{0'}\eta(g)_{s-1}=0.
\end{equation}
Combining the condition \eqref{eq:thm main: x notin X:mid:assume of f and g} and equation \eqref{eq:LA eta}, the above identity follows directly.

Finally, we verify that $D^{(k+2)}h=0$. Define
\[H:=\eta(h),\qquad F:=\eta(f),\qquad G:=\eta(g). \]
Combining with \eqref{eq:thm main: x notin X:mid:1}, we obtain
\begin{equation*}
\begin{aligned}
d^{0'}H_s+d^{1'}H_{s+1}
&=d^{0'}\left(\frac{1}{|q_0|^2}\big(z_0^{1'}F_{s}+z_1^{1'}G_s\big)\right)+d^{1'}\left(\frac{1}{|q_0|^2}\big(z_0^{1'}F_{s+1}+z_1^{1'}G_{s+1}\big)\right)\\
&=\frac{z_0^{1'}}{|q_0|^2}\big(d^{0'}F_s+d^{1'}F_{s+1}\big)+\frac{z_1^{1'}}{|q_0|^2}\big(d^{0'}G_s+d^{1'}G_{s+1}\big)\\
&\quad +\left(d^{0'}\frac{z_0^{1'}}{|q_0|^2}\right)\wedge F_s+\left(d^{0'}\frac{z_1^{1'}}{|q_0|^2}\right)\wedge G_s\\
&\quad+\left(d^{1'}\frac{z_0^{1'}}{|q_0|^2}\right)\wedge F_{s+1}+\left(d^{1'}\frac{z_1^{1'}}{|q_0|^2}\right)\wedge G_{s+1}.
\end{aligned}
\end{equation*}
Since $D^{(k+1)}f=D^{(k+1)}g=0$, we have
\[d^{0'}F_s+d^{1'}F_{s+1}=0,\qquad d^{0'}G_s+d^{1'}G_{s+1}=0. \]
Hence, it suffices to prove that
\[\left(d^{0'}\frac{z_0^{1'}}{|q_0|^2}\right)\wedge F_s+\left(d^{0'}\frac{z_1^{1'}}{|q_0|^2}\right)\wedge G_s+\left(d^{1'}\frac{z_0^{1'}}{|q_0|^2}\right)\wedge F_{s+1}+\left(d^{1'}\frac{z_1^{1'}}{|q_0|^2}\right)\wedge G_{s+1}=0.\]
Recall the definitions of $z_A^{A'}$ in \eqref{eq:def:zAAprime} and of $\nabla_A^{A'}$ in \eqref{preliminary:nabla AA prime}. We have
\begin{equation*}
\nabla_A^{A'}|q_0|^2=-2z_A^{A'}, \qquad A=0,1,\quad A'=0',1'.
\end{equation*}
Combining this identity with Lemma \ref{lem:computation z}, we obtain
\begin{align*}
&|q_0|^4\left(\left(d^{0'}\frac{z_0^{1'}}{|q_0|^2}\right)\wedge F_s+\left(d^{0'}\frac{z_1^{1'}}{|q_0|^2}\right)\wedge G_s+\left(d^{1'}\frac{z_0^{1'}}{|q_0|^2}\right)\wedge F_{s+1}+\left(d^{1'}\frac{z_1^{1'}}{|q_0|^2}\right)\wedge G_{s+1}\right)\\
&=2\big(z_0^{1'}(z_0^{0'}\omega^0+z_1^{0'}\omega^1)+|q_0|^2\omega^1\big)\wedge F_s
+2\big(z_1^{1'}(z_0^{0'}\omega^0+z_1^{0'}\omega^1)-|q_0|^2\omega^0\big)\wedge G_s\\
&\quad +2z_0^{1'}(z_0^{1'}\omega^0+z_1^{1'}\omega^1)\wedge F_{s+1}+2z_1^{1'}(z_0^{1'}\omega^0+z_1^{1'}\omega^1)\wedge G_{s+1}\\
&=2\big(z_0^{1'}\omega^0+z_1^{1'}\omega^1\big)\wedge \big(z_0^{0'}F_s+z_1^{0'}G_s+z_0^{1'}F_{s+1}+z_1^{1'}G_{s+1}\big)=0.
\end{align*}
The final equality holds due to \eqref{eq:thm main: x notin X:mid:2}.

\textbf{Step 3: The surjectivity of $\mathscr{L}^{(k)}_x$.}

Let $U$ be a convex neighborhood of $x$ with $U\cap X=\emptyset$ and let $h\in \mathcal{R}^{(k)}(U)$.

  Recall the definitions of $\mathscr{L}^{(k)}$ and equation \eqref{eq:LA eta}. We have
\begin{equation*}
\begin{aligned}
\begin{pmatrix}
\eta\big(\mathscr{L}^{(k)}(f,g)^T\big)_0 \\
\vdots \\
\eta\big(\mathscr{L}^{(k)}(f,g)^T\big)_{k}
\end{pmatrix}
&=
\begin{pmatrix}
z_0^{0'} & z_0^{1'}  &  & \\
& z_0^{0'} & z_0^{1'}  & \\
& & \ddots & \ddots
\end{pmatrix}
\begin{pmatrix}
\eta(f)_0 \\
\vdots \\
\eta(f)_{k+1}
\end{pmatrix}\\
&\quad +
\begin{pmatrix}
z_1^{0'} & z_1^{1'}  &  & \\
& z_1^{0'} & z_1^{1'}  & \\
& & \ddots & \ddots
\end{pmatrix}
\begin{pmatrix}
\eta(g)_0 \\
\vdots \\
\eta(g)_{k+1}
\end{pmatrix}.
\end{aligned}
\end{equation*}
From the definition of $z_A^{A'}$ given in \eqref{eq:def:zAAprime}, it follows that
\begin{equation*}
\begin{aligned}
&z_0^{0'}\overline{z_0^{0'}}+z_0^{1'}\overline{z_0^{1'}}+z_1^{0'}\overline{z_1^{0'}}+z_1^{1'}\overline{z_1^{1'}}=2|q_0|^2,\\
&z_0^{1'}\overline{z_0^{0'}}+z_1^{1'}\overline{z_1^{0'}}=0. \\
\end{aligned}
\end{equation*}
  This implies that
\begin{equation*}
\begin{aligned}
2|q_0|^2\mathrm{Id}_{k+1}
&=
\begin{pmatrix}
z_0^{0'} & z_0^{1'}  &  & \\
& z_0^{0'} & z_0^{1'}  & \\
& & \ddots & \ddots
\end{pmatrix}
\begin{pmatrix}
\overline{z_0^{0'}} & \overline{z_0^{1'}}  &  & \\
& \overline{z_0^{0'}} & \overline{z_0^{1'}}  & \\
& & \ddots & \ddots
\end{pmatrix}^{\!T} \\
&\quad +
\begin{pmatrix}
z_1^{0'} & z_1^{1'}  &  & \\
& z_1^{0'} & z_1^{1'}  & \\
& & \ddots & \ddots
\end{pmatrix}
\begin{pmatrix}
\overline{z_1^{0'}} & \overline{z_1^{1'}}  &  & \\
& \overline{z_1^{0'}} & \overline{z_1^{1'}}  & \\
& & \ddots & \ddots
\end{pmatrix}^{\!T}.
\end{aligned}
\end{equation*}

We define two smooth functions $F,G\in C^\infty(U,\mathbb{C}^{k+2})$ via
\begin{align*}
&F=\frac{1}{2|q_0|^2}\begin{pmatrix}
\overline{z_0^{0'}} & \overline{z_0^{1'}}  &  & \\
& \overline{z_0^{0'}} & \overline{z_0^{1'}}  & \\
& & \ddots & \ddots
\end{pmatrix}^{\!T} \eta(h),\\
&G=\frac{1}{2|q_0|^2}\begin{pmatrix}
\overline{z_1^{0'}} & \overline{z_1^{1'}}  &  & \\
& \overline{z_1^{0'}} & \overline{z_1^{1'}}  & \\
& & \ddots & \ddots
\end{pmatrix}^{\!T}\eta(h).
\end{align*}
Let $f:=\eta^{-1}(F)$, $g:=\eta^{-1}(G)$, $H:=\eta(h)$ and define $H_{-1}=H_{k+1}=0$. The components of $F$ and $G$ can be explicitly expressed in terms of $H$ as
\begin{equation}\label{eq:thm main: x notin X:last:def f g}
\begin{aligned}
&F_s=\frac{1}{2|q_0|^2}(\overline{z_0^{1'}}H_{s-1}+\overline{z_0^{0'}}H_s)=\frac{1}{2|q_0|^2}(-z_1^{0'}H_{s-1}+z_1^{1'}H_s),\\
&G_s=\frac{1}{2|q_0|^2}(\overline{z_1^{1'}}H_{s-1}+\overline{z_1^{0'}}H_s)=\frac{1}{2|q_0|^2}(z_0^{0'}H_{s-1}-z_0^{1'}H_s).
\end{aligned}
\end{equation}
In addition, the functions $f$ and $g$ obey the identity
\begin{equation*}
\mathscr{L}^{(k)}(f,g)^T=h.
\end{equation*}
However, $f$ and $g$ are generally not $(k+1)$-regular, which necessitates suitable adjustments. Let $w\in C^\infty\big(U,\odot^{k+2}\mathbb{C}^{2}\big)$. We consider the expressions
\[f+L_1^{(k+2)}w,\qquad g-L_0^{(k+2)}w.\]
These two quantities still satisfy
\begin{equation*}
\mathscr{L}^{(k)}\bigl(f+L_1^{(k+2)}w,g-L_0^{(k+2)}w\bigr)^T=h.
\end{equation*}
Our remaining task is to construct a suitable $w$ such that $f+L_1^{(k+2)}w$ and $g-L_0^{(k+2)}w$ become $(k+1)$-regular.

The proof proceeds by considering two cases.

\textbf{Case $k\geqslant 1$.}
In this case, we have $f,g\in C^\infty\big(U,\odot^{k+1}\mathbb{C}^2\big)$, and
\[
\mathscr{D}_0^{(k+1)}f,\,\mathscr{D}_0^{(k+1)}g\in C^\infty\big(U,\odot^{k}\mathbb{C}^2\otimes \Lambda^1\mathbb{C}^{2n}\big).
\]
By Proposition \ref{prop:Lj and k CF}, we obtain
\begin{equation}\label{eq:thm main: x notin X:last:1}
L_0^{(k)}\mathscr{D}_0^{(k+1)}f+L_1^{(k)}\mathscr{D}_0^{(k+1)}g
=\mathscr{D}_0^{(k)}\big(L_0^{(k+1)}f+L_1^{(k+1)}g\big)=0.
\end{equation}

From the proof of the exactness of the sequence
\[
\mathcal{R}^{(k+2)}_x\xrightarrow{\widetilde{\mathscr{L}}^{(k)}}\mathcal{R}^{(k+1)}_x\oplus\mathcal{R}^{(k+1)}_x\xrightarrow{\mathscr{L}^{(k)}}\mathcal{R}_{x}^{(k)},
\]
there exists $v\in C^\infty\big(U,\odot^{k+1}\mathbb{C}^2\otimes\Lambda^1\mathbb{C}^{2n}\big)$ satisfying
\begin{equation}\label{eq:thm main: x notin X:last:v}
-L_1^{(k+1)}v=\mathscr{D}_0^{(k+1)}f,\qquad L_0^{(k+1)}v=\mathscr{D}_0^{(k+1)}g.
\end{equation}

Combining the definition of $D^{(k)}$ in \eqref{eq:def Dk} with Proposition \ref{prop:Lj and k CF}, we deduce
\begin{equation*}
L_1^{(k)}\mathscr{D}_1^{(k+2)}v
=L_1^{(k)}D^{(k+1)}v
=D^{(k)}L_1^{(k+1)}v
=-\mathscr{D}_1^{(k+1)}\mathscr{D}_0^{(k+1)}f=0
\end{equation*}
and
\begin{equation*}
L_0^{(k)}\mathscr{D}_1^{(k+2)}v
=L_0^{(k)}D^{(k+1)}v
=D^{(k)}L_0^{(k+1)}v
=\mathscr{D}_1^{(k+1)}\mathscr{D}_0^{(k+1)}g=0.
\end{equation*}

Applying the proof of injectivity of $\widetilde{\mathscr{L}}^{(k)}$, we conclude
\[
\mathscr{D}_1^{(k+2)}v=0.
\]

Recall that $U$ is convex and $H^1(U,\mathcal{R}^{(k+2)}) = 0$. Then there exists a solution to the equation
\begin{equation*}
\mathscr{D}_0^{(k+2)}u = v.
\end{equation*}
Let $w$ be such a solution. Combining Proposition \ref{prop:Lj and k CF} with \eqref{eq:thm main: x notin X:last:v}, we obtain
\begin{equation}\label{eq:thm main: x notin X:last:case k large: final}
\begin{aligned}
\mathscr{D}_0^{(k+1)}\big(f+L_1^{(k+2)}w\big) &= \mathscr{D}_0^{(k+1)}f + L_1^{(k+1)}\mathscr{D}_0^{(k+2)}w = \mathscr{D}_0^{(k+1)}f + L_1^{(k+1)}v = 0,\\
\mathscr{D}_0^{(k+1)}\big(g-L_0^{(k+2)}w\big) &= \mathscr{D}_0^{(k+1)}g - L_0^{(k+1)}\mathscr{D}_0^{(k+2)}w = \mathscr{D}_0^{(k+1)}g - L_0^{(k+1)}v = 0.
\end{aligned}
\end{equation}
Consequently, we have
\begin{equation*}
\mathscr{L}^{(k)}\bigl(f+L_1^{(k+2)}w,g-L_0^{(k+2)}w\bigr)^T = h.
\end{equation*}
Moreover, the functions $f+L_1^{(k+2)}w$ and $g-L_0^{(k+2)}w$ are $(k+1)$-regular.

\textbf{Case $k=0$.}
For $k=0$, the proof differs from that for $k\geqslant 1$, since the operator $D^{(0)}$ is not well-defined, which means Proposition \ref{prop:Lj and k CF} is inapplicable. Consequently, equation \eqref{eq:thm main: x notin X:last:1} no longer holds.

However, we note that $\mathscr{D}_0^{(k+1)}f,\mathscr{D}_0^{(k+1)}g\in C^\infty(U,\Lambda^1\mathbb{C}^{2n})$. Accordingly, the function $v$ defined in \eqref{eq:thm main: x notin X:last:v} satisfies
\begin{equation*}
\begin{pmatrix}
-z_1^{0'} & -z_1^{1'}\\
z_0^{0'} & z_0^{1'}
\end{pmatrix}
\begin{pmatrix}
v_{0'}\\
v_{1'}
\end{pmatrix}=\begin{pmatrix}
\mathscr{D}_0^{(1)}f\\
\mathscr{D}_0^{(1)}g
\end{pmatrix}.
\end{equation*}
Since the matrix
\[
\begin{pmatrix}
-z_1^{0'} & -z_1^{1'}\\
z_0^{0'} & z_0^{1'}
\end{pmatrix}
\]
is invertible, we can explicitly define the function $v$ from \eqref{eq:thm main: x notin X:last:v} as follows:
\begin{equation}\label{eq:thm main: x notin X:last:case k0 v}
\begin{pmatrix}
v_{0'}\\
v_{1'}
\end{pmatrix}=\frac{1}{|q_0|^2}\begin{pmatrix}
z_0^{1'} & z_1^{1'}\\
-z_0^{0'} & -z_1^{0'}
\end{pmatrix}\begin{pmatrix}
\mathscr{D}_0^{(1)}f\\
\mathscr{D}_0^{(1)}g
\end{pmatrix}.
\end{equation}

The final task is to prove that $\mathscr{D}_1^{(2)}v=d^{0'}v_{0'}+d^{1'}v_{1'}=0$. Combining the definition of $v$ in \eqref{eq:thm main: x notin X:last:case k0 v}, the definitions of $f$ and $g$ in \eqref{eq:thm main: x notin X:last:def f g}, along with Lemma \ref{lem:computation z}, we obtain
\begin{align*}
\begin{pmatrix}
v_{0'}\\
v_{1'}
\end{pmatrix}
&=
\begin{pmatrix}
\frac{1}{|q_0|^2}d^{1'}h
-\frac{2z_0^{1'}}{|q_0|^4}\omega^0\wedge h
-\frac{2z_1^{1'}}{|q_0|^4}\omega^1\wedge h \\
-\frac{1}{|q_0|^2}d^{0'}h
+\frac{2z_0^{0'}}{|q_0|^4}\omega^0\wedge h
+\frac{2z_1^{0'}}{|q_0|^4}\omega^1\wedge h
\end{pmatrix}\\
&=
\begin{pmatrix}
\frac{1}{|q_0|^2}d^{1'}h
-\left(d^{1'}\frac{1}{|q_0|^2}\right)\wedge h \\
- \frac{1}{|q_0|^2}d^{0'}h
+\left(d^{0'}\frac{1}{|q_0|^2}\right)\wedge h
\end{pmatrix}.
\end{align*}
Notice that $d^{0'}d^{1'}\dfrac{1}{|q_0|^2}=d^{0'}d^{1'}h=0$. This immediately yields
\[d^{0'}v_{0'}+d^{1'}v_{1'}=0.\]
Furthermore, since the set $U$ is convex, we have $H^1(U,\mathcal{R}^{(2)})=0$. Thus there exists a function $w$ satisfying $\mathscr{D}_0^{(2)}w=v$. Combining with equation \eqref{eq:thm main: x notin X:last:case k large: final}, we obtain
\begin{equation*}
\mathscr{L}^{(k)}\bigl(f+L_1^{(k+2)}w,g-L_0^{(k+2)}w\bigr)^T = h
\end{equation*}
where the functions $f+L_1^{(k+2)}w$ and $g-L_0^{(k+2)}w$ are $(k+1)$-regular.

This completes the proof of case $x\notin X$.

\begin{remark}
   We would like to point out that the proof for the case $x\notin X$ employs multiple tools from homological algebra, including the connecting map. In Section \ref{section:long exact sequence}, we extend the above argument to establish the existence of a long exact sequence.
\end{remark}

\subsection{Proof of Theorem \ref{thm:main}: Case $x\in X$}

The proof for the case $x\in X$ relies heavily on the Taylor expansion of $k$-regular functions. We therefore introduce the Taylor expansion of $k$-regular functions presented in \cite{KW13}.

Let $s,j\in \mathbb{Z}$, $\mathbf{m}=(m_0,\ldots,m_{2n-1})\in \mathbb{Z}^{2n}_{\geqslant0}$ and define
\[|\mathbf{m}|=m_0+\cdots +m_{2n-1}.\]
In \cite{KW13}, Kang and Wang defined the polynomials $P^{\mathbf{m}}_{s,j}$ by
\begin{equation}\label{eq:def:P m sj}
P^{\mathbf{m}}_{s,j}:=\frac{1}{2\pi i}\oint_{|r|=1}r^{s-j}\prod_{A=0}^{2n-1} \big(z^{A0'}+z^{A1'}r\big)^{m_{A}}dr.
\end{equation}
Moreover, for any $k$-regular function $f$, there exists a series expansion of $f$ in terms of the polynomials $P^{\mathbf{m}}_{s,j}$.

\begin{theorem}\label{thm:series expansion}\cite[Theorem 1.2]{KW13}
Let $f$ be a $k$-regular function defined on a neighborhood of the origin $0$. Then there exists a neighborhood $B_0(\varepsilon)$ of $0$ and complex coefficients $C_{\mathbf{m}}^{j}$ such that the series
\[
\sum_{\mathbf{m}\in \mathbb{Z}_{\geqslant 0}^{2n}}\sum_{j=1}^{|\mathbf{m}|+s+1}C_{\mathbf{m}}^j P^{\mathbf{m}}_{s,j}
\]
converges uniformly to $\eta(f)_s$ on $B_0(\varepsilon)$ for each integer $s=0,\ldots,k$. Furthermore, each $P_{s,j}^{\mathbf{m}}$ takes the form
\begin{equation}\label{eq:rep:P m sj}
P_{s,j}^{\mathbf{m}}=\sum_{\substack{\mathbf{r},\mathbf{m}-\mathbf{r}\in \mathbb{Z}^{2n}_{\geqslant 0}\\ |\mathbf{r}|=j-1-s}}\left(\prod_{A}\binom{m_A}{r_A}\right)\left(\prod_{A}(z^{A0'})^{m_A-r_A}(z^{A1'})^{r_A}\right).
\end{equation}
\end{theorem}

Some remarks and lemmas concerning the above polynomials and their coefficients are given below.

\begin{remark}
For integers $t > |\mathbf{m}|$ or $t < 0$, we define the coefficient of $r^t$ in the polynomial
\[
f(r)=\prod_{A=0}^{2n-1} \big(z^{A0'}+z^{A1'}r\big)^{m_{A}}
\]
to be zero. By the definition of $P_{s,j}^{\mathbf{m}}$ \eqref{eq:rep:P m sj}, the polynomial $P_{s,j}^{\mathbf{m}}$ equals the coefficient of $r^{j-s-1}$ in $f(r)$.

Furthermore, we set
\[
C_{\mathbf{m}}^{j}:=0
\]
if $j-s \leqslant 0$ or $j-s \geqslant |\mathbf{m}|+2$. With these notations, the series stated in Theorem \ref{thm:series expansion} can be rewritten as
\[
\sum_{\mathbf{m}\in \mathbb{Z}_{\geqslant 0}^{2n}}\sum_{j\in \mathbb{Z}}C_{\mathbf{m}}^j P^{\mathbf{m}}_{s,j}.
\]
\end{remark}

\begin{lemma}\label{lem:L action on P sj}
For all $s,j\in\mathbb{Z}$, the identities
\begin{equation*}
z_{2t}^{0'}P^{\mathbf{m}}_{s,j}+z_{2t}^{1'}P^{\mathbf{m}}_{s+1,j}=P^{\mathbf{m}+\mathbf{e}_{2t+1}}_{s,j},\quad z_{2t+1}^{0'}P^{\mathbf{m}}_{s,j}+z_{2t+1}^{1'}P^{\mathbf{m}}_{s+1,j}=-P^{\mathbf{m}+\mathbf{e}_{2t}}_{s,j}
\end{equation*}
hold. Here $\mathbf{e}_{A}$ stands for the standard basis vector of $\mathbb{Z}^{2n}$ whose $A$-th coordinate is $1$ and all remaining coordinates are $0$.
\end{lemma}
\begin{proof}
We first observe the identities
\begin{align*}
z_{2t}^{0'}P^{\mathbf{m}}_{s,j}+z_{2t}^{1'}P^{\mathbf{m}}_{s+1,j}
&=z^{(2t+1)0'}P^{\mathbf{m}}_{s,j}+z^{(2t+1)1'}P^{\mathbf{m}}_{s+1,j},\\
z_{2t+1}^{0'}P^{\mathbf{m}}_{s,j}+z_{2t+1}^{1'}P^{\mathbf{m}}_{s+1,j}
&=-z^{(2t)0'}P^{\mathbf{m}}_{s,j}-z^{(2t)1'}P^{\mathbf{m}}_{s+1,j},
\end{align*}
where the left-hand sides of these two equalities are precisely the coefficients of $r^{j-s-1}$ in
\[
\big(z^{(2t+1)0'}+z^{(2t+1)1'}r\big)\prod_{A=0}^{2n-1} \big(z^{A0'}+z^{A1'}r\big)^{m_{A}}
\quad\text{and}\quad
-\big(z^{(2t)0'}+z^{(2t)1'}r\big)\prod_{A=0}^{2n-1} \big(z^{A0'}+z^{A1'}r\big)^{m_{A}},
\]
in that order. Combining the definition of $P_{\mathbf{m}}^{s,j}$, we obtain the identities
\[
z_{2t}^{0'}P^{\mathbf{m}}_{s,j}+z_{2t}^{1'}P^{\mathbf{m}}_{s+1,j}=P^{\mathbf{m}+\mathbf{e}_{2t+1}}_{s,j},\quad z_{2t+1}^{0'}P^{\mathbf{m}}_{s,j}+z_{2t+1}^{1'}P^{\mathbf{m}}_{s+1,j}=-P^{\mathbf{m}+\mathbf{e}_{2t}}_{s,j}.
\]
\end{proof}

The following lemma reveals the connection between the series expansion given in Theorem \ref{thm:series expansion} and the Taylor expansion, and provides an estimate for the coefficient $C_{\mathbf{m}}^{j}$ and the polynomial $P^{\mathbf{m}}_{s,j}$.

\begin{lemma}\label{lem:estimate of C mj and P m sj}
Let $\varepsilon>0$ be sufficiently small. Suppose $f$ is a $k$-regular function defined on $B_0(\varepsilon)$, and let $C_{\mathbf{m}}^j$ and $P^{\mathbf{m}}_{s,j}$ be the quantities introduced in Theorem \ref{thm:series expansion}. The power series expansion of $f$ stated in Theorem \ref{thm:series expansion} coincides with the Taylor series of $f$ up to a rearrangement of summation indices. Furthermore, we establish the following estimates for the coefficient $C_{\mathbf{m}}^{j}$ and the polynomial $P^{\mathbf{m}}_{s,j}$:
\[
|C_{\mathbf{m}}^j|\leqslant C\left(\frac{2^{4n+5}n^2e}{\varepsilon}\right)^{|\mathbf{m}|}
\int_{B_{0}(\varepsilon)}\frac{|\eta(f)_s|}{\sigma_{4n}\varepsilon^{4n}}\,dV.
\]
\[
|P^{\mathbf{m}}_{s,j}|\leqslant \sum_{\substack{\mathbf{r},\mathbf{m}-\mathbf{r}\in \mathbb{Z}^{2n}_{\geqslant 0}\\ |\mathbf{r}|=j-1-s}}
(2\varepsilon)^{|\mathbf{m}|}\prod_{A}\binom{m_A}{r_A}.
\]
\end{lemma}
\begin{proof}
Note that every component of a $k$-regular function is harmonic and admits a Taylor expansion at the origin:
\begin{equation}\label{eq:Taylor expansion}
\eta(f)_s=\sum_{m=0}^\infty\sum_{|\alpha|=m}\frac{\big(D_\alpha\eta(f)_s\big)(0) \, x^{\alpha}}{\alpha!}.
\end{equation}
It is well known that there exists a Cauchy-type estimate for derivatives of harmonic functions:
\begin{align*}
\frac{1}{\alpha!}\big|(D_\alpha\eta(f)_s)(0)\big|
&\leqslant \frac{(2^{4n+3}\,n\,|\alpha|)^{|\alpha|}}{\alpha!\,\varepsilon^{|\alpha|}}
\int_{B_{0}(\varepsilon)}\frac{|\eta(f)_s|}{\sigma_{4n}\varepsilon^{4n}}\,dV\\
&\leqslant C\left(\frac{2^{4n+5}n^2e}{\varepsilon}\right)^{|\alpha|}
\int_{B_{0}(\varepsilon)}\frac{|\eta(f)_s|}{\sigma_{4n}\varepsilon^{4n}}\,dV,
\end{align*}
where $\sigma_{4n}$ stands for the volume of the unit ball $B_0(1)$, and $C$ is a constant depending only on the dimension. This estimate is presented in Theorems 7 and 10 of Section 2.2, Chapter 2 in \cite{Evans98}. Then  the series
\[
\sum_{m=0}^\infty\sum_{|\alpha|=m}\frac{\left|\big(D_\alpha\eta(f)_s\big)(0) \, x^{\alpha}\right|}{\alpha!}
\]
converges uniformly on $B_0\!\left(C'\varepsilon\right)$ for some constant $C'$. Consequently, the series \eqref{eq:Taylor expansion} admits interchange of the summation order.

Indeed, we claim that the polynomial vector
\[
\eta^{-1}\begin{pmatrix}
\vdots \\
P^{\mathbf{m}}_{s,j}\\
\vdots
\end{pmatrix}
\]
constitutes a basis for the space of $k$-regular polynomials of degree $|\mathbf{m}|$. The reasoning is as follows: for any fixed index $s$, each monomial of the form
\[
\prod_{A,A'} z^{AA'}
\]
occurs in at most one polynomial $P^{\mathbf{m}}_{s,j}$ across all admissible values of $j$.

We may now reorganize the summation indices to recast the Taylor expansion of $\eta(f)_s$ given in \eqref{eq:Taylor expansion} into a sum of terms of the form $C_{\mathbf{m}}^j P^{\mathbf{m}}_{s,j}$, as shown below:
\[
\begin{aligned}
\eta(f)_s(z)
&=\sum_{\alpha\in \mathbb{Z}^{4n}_{\geqslant 0}}\prod_{A,A'}\frac{1}{2^{|\alpha|}}\cdot\frac{(\overline{\nabla_{AA'}})^{\alpha_{AA'}}\eta(f)_s(0)\cdot (z^{AA'})^{\alpha_{AA'}}}{\alpha_{AA'}!}\\
&=\sum_{\mathbf{m}\in\mathbb{Z}^{2n}_{\geqslant 0}}\sum_{\substack{\mathbf{r},\mathbf{m}-\mathbf{r}\in \mathbb{Z}^{2n}_{\geqslant 0}\\ |\mathbf{r}|=j-1-s}}
\frac{(\overline{\nabla_{A0'}})^{m_A-r_A}(\overline{\nabla_{A1'}})^{r_A}\eta(f)_s(0)}{2^{|\mathbf{m}|}\,r_A!\,(m_A-r_A)!}
\,(z^{A0'})^{m_A-r_A}(z^{A1'})^{r_A}\\
&=\sum_{\mathbf{m}\in \mathbb{Z}_{\geqslant 0}^{2n}}\sum_{j\in\mathbb{Z}}C_{\mathbf{m}}^j P^{\mathbf{m}}_{s,j},
\end{aligned}
\]
where the coefficients $C_{\mathbf{m}}^j$ are defined via
\begin{equation}\label{eq:C mj and Taylor expansion}
C_{\mathbf{m}}^j=\frac{1}{2^{|\mathbf{m}|}\prod_{A}m_A!}(\overline{\nabla_{A0'}})^{m_A-r_A}(\overline{\nabla_{A1'}})^{r_A}\eta(f)_s(0).
\end{equation}
The formula above is independent of the choice of $\mathbf{r}$ due to the $k$-regularity of $f$.

Then the estimate of $C_{\mathbf{m}}^j$ follows from Theorems 7 and 10 in Section 2.2 of \cite{Evans98}, whereas the estimate of $P^{\mathbf{m}}_{s,j}$ follows immediately from equation \eqref{eq:rep:P m sj}.
\end{proof}

We can complete the remaining proof of Theorem \ref{thm:main} by using the series expansion.

\emph{Proof of Theorem \ref{thm:main}: case $x\in X$.}

It should be noted that $k$-regular functions possess translation invariance. Hence, we may assume $x=0$ without loss of generality.

Let $B_0(\varepsilon)$ denote a sufficiently small neighbourhood of the origin. We proceed with the proof by discussing several steps separately.

\textbf{Step 1: The injectivity of $\widetilde{\mathscr{L}}^{(k)}$.}

Assume that $\widetilde{\mathscr{L}}^{(k)} h= 0$ on $U=B_0(\varepsilon)$. From the argument for the case $x \notin X \cap \Omega$, we conclude that $h(x) = 0$ for all $x \in U \cap \{q_0 \neq 0\}$. Since $h(x)$ is harmonic and $U \cap \{q_0 \neq 0\}$ is dense in $U$, it follows that $h = 0$ on $U$.

\textbf{Step 2: The exactness of $\mathcal{R}^{(k+2)}_x\xrightarrow{\widetilde{\mathscr{L}}^{(k)}}\mathcal{R}^{(k+1)}_x\oplus\mathcal{R}^{(k+1)}_x\xrightarrow{\mathscr{L}^{(k)}}\mathcal{R}_{x}^{(k)}$.}

Let $f,g\in \mathcal{R}^{(k+1)}(B_0(\varepsilon))$ satisfy
\[
L_0^{(k+1)}f+L_1^{(k+1)}g=0.
\]
It suffices to find some $h\in \mathcal{R}^{(k+2)}(B_0(\varepsilon^2))$ such that
\[
-L_1^{(k+2)}h=f,\qquad L_0^{(k+2)}h=g.
\]
By the proof case $x\notin X\cap \Omega$, away from $X$, the function $h$ necessarily obeys:
\begin{equation*}
\begin{pmatrix}
\eta(h)_s\\
\eta(h)_{s+1}
\end{pmatrix}=
\frac{1}{|q_0|^2}
\begin{pmatrix}
z_0^{1'} & z_1^{1'}\\
-z_0^{0'} & -z_1^{0'}
\end{pmatrix}
\begin{pmatrix}
\eta(f)_s\\
\eta(g)_s
\end{pmatrix}.
\end{equation*}
It remains to verify that the functions
\begin{equation}\label{eq:proof of thm main:case x in X:exactness:def of function h}
\frac{1}{|q_0|^2}
\begin{pmatrix}
z_0^{1'} & z_1^{1'}\\
-z_0^{0'} & -z_1^{0'}
\end{pmatrix}
\begin{pmatrix}
\eta(f)_s\\
\eta(g)_s
\end{pmatrix}
\end{equation}
admit smooth extensions across $X$ and are $(k+2)$-regular at every point $x\in B_0(\varepsilon^2)\cap X$. Assume that both $f$ and $g$ admit series expansions on $U$ of the following form:
\begin{align*}
\eta(f)_s&=\sum_{\mathbf{m}\in \mathbb{Z}_{\geqslant 0}^{2n}}\sum_{j\in\mathbb{Z}}{C}_{f,\mathbf{m}}^j P^{\mathbf{m}}_{s,j},\\
\eta(g)_s&=\sum_{\mathbf{m}\in \mathbb{Z}_{\geqslant 0}^{2n}}\sum_{j\in\mathbb{Z}}{C}_{g,\mathbf{m}}^j P^{\mathbf{m}}_{s,j}.
\end{align*}
Moreover, the above series expansions converge uniformly on $B_0(\varepsilon)$.
Then the condition $L_0^{(k+1)}f + L_1^{(k+1)}g = 0$, combined with equation \eqref{eq:LA eta} and Lemma \ref{lem:L action on P sj}, yields that for every $s = 0,1,\dots,k$, the series
\[
\sum_{\mathbf{m}\in \mathbb{Z}_{\geqslant 0}^{2n}}\sum_{j\in\mathbb{Z}}\left(C_{f,\mathbf{m}}^j P^{\mathbf{m}+\mathbf{e}_1}_{s,j}-C_{g,\mathbf{m}}^j P^{\mathbf{m}+\mathbf{e}_0}_{s,j}\right)
\]
converges uniformly to $0$ on $U$. We next consider the series
\[
\sum_{|i|=0}^{\infty}\sum_{\mathbf{m}\in \mathbb{Z}_{\geqslant 0}^{2n},\,|\mathbf{m}|=i}\sum_{j\in\mathbb{Z}}\left(C_{f,\mathbf{m}}^j P^{\mathbf{m}+\mathbf{e}_1}_{s,j}-C_{g,\mathbf{m}}^j P^{\mathbf{m}+\mathbf{e}_0}_{s,j}\right).
\]
The polynomial
\[
\sum_{\mathbf{m}\in \mathbb{Z}_{\geqslant 0}^{2n},\,|\mathbf{m}|=i}\sum_{j\in\mathbb{Z}}\left(C_{f,\mathbf{m}}^j P^{\mathbf{m}+\mathbf{e}_1}_{s,j}-C_{g,\mathbf{m}}^j P^{\mathbf{m}+\mathbf{e}_0}_{s,j}\right)
\]
is the $i$-th homogeneous part of above series. It follows from Proposition 1.30 in \cite{ABR01} that
\[
\sum_{\mathbf{m}\in \mathbb{Z}_{\geqslant 0}^{2n},\,|\mathbf{m}|=i}\sum_{j\in\mathbb{Z}}\left(C_{f,\mathbf{m}}^j P^{\mathbf{m}+\mathbf{e}_1}_{s,j}-C_{g,\mathbf{m}}^j P^{\mathbf{m}+\mathbf{e}_0}_{s,j}\right)=0
\]
holds for all $i=0,1,\dots$. Recall that for fixed $s$, the set $\{P_{s,j}^{\mathbf{m}}\}$ is linearly independent. We have
\begin{equation}\label{eq:proof of thm main:case x in X:exactness:cof 1}
C_{f,\mathbf{m}'}^{j}=C_{g,\mathbf{m'}}^{j}=0
\end{equation}
for all $\mathbf{m}'$ such that $m_0'=0$ or $m_1'=0$. Furthermore,
\begin{equation}\label{eq:proof of thm main:case x in X:exactness:cof 2}
C_{f,\mathbf{m}-\mathbf{e}_1}^{j}-C_{g,\mathbf{m}-\mathbf{e}_0}^{j}=0
\qquad \text{whenever} \qquad \mathbf{m}-\mathbf{e}_0,\mathbf{m}-\mathbf{e}_1\in \mathbb{Z}^{2n}_{\geqslant 0}.
\end{equation}

It suffices to show that the function \eqref{eq:proof of thm main:case x in X:exactness:def of function h} belongs to $C^1$ on $B_0(\varepsilon^2)$. It follows from equations \eqref{eq:proof of thm main:case x in X:exactness:cof 1} and \eqref{eq:proof of thm main:case x in X:exactness:cof 2} that
\begin{align*}
    z_0^{1'} \eta(f)_s + z_1^{1'}\eta(g)_s&=\sum_{\mathbf{m}\in \mathbb{Z}_{\geqslant 0}^{2n}}\sum_{j\in\mathbb{Z}}{C}_{f,\mathbf{m}}^j z_0^{1'}P^{\mathbf{m}}_{s,j}+ \sum_{\mathbf{m}\in \mathbb{Z}_{\geqslant 0}^{2n}}\sum_{j\in\mathbb{Z}}{C}_{g,\mathbf{m}}^j z_1^{1'}P^{\mathbf{m}}_{s,j}\\
    &=\sum_{\mathbf{m}-\mathbf{e}_0-\mathbf{e}_1\in \mathbb{Z}_{\geqslant 0}^{2n}}\sum_{j\in\mathbb{Z}}{C}_{f,\mathbf{m}-\mathbf{e}_1}^j z_0^{1'}P^{\mathbf{m}-\mathbf{e}_1}_{s,j}\\
    &\quad + \sum_{\mathbf{m}-\mathbf{e}_0-\mathbf{e}_1\in \mathbb{Z}_{\geqslant 0}^{2n}}\sum_{j\in\mathbb{Z}}{C}_{g,\mathbf{m}-\mathbf{e}_0}^j z_1^{1'}P^{\mathbf{m}-\mathbf{e}_0}_{s,j},
    \end{align*}
    and
    \begin{align*}
z_0^{0'}\eta(f)_s +z_1^{0'}\eta(g)_s&=\sum_{\mathbf{m}\in \mathbb{Z}_{\geqslant 0}^{2n}}\sum_{j\in\mathbb{Z}}{C}_{f,\mathbf{m}}^j z_0^{0'}P^{\mathbf{m}}_{s,j}+\sum_{\mathbf{m}\in \mathbb{Z}_{\geqslant 0}^{2n}}\sum_{j\in\mathbb{Z}} {C}_{g,\mathbf{m}}^j z_1^{0'}P^{\mathbf{m}}_{s,j}\\
&=\sum_{\mathbf{m}-\mathbf{e}_0-\mathbf{e}_1\in \mathbb{Z}_{\geqslant 0}^{2n}}\sum_{j\in\mathbb{Z}}{C}_{f,\mathbf{m}-\mathbf{e}_1}^j z_0^{0'}P^{\mathbf{m}-\mathbf{e}_1}_{s,j}\\
&\quad + \sum_{\mathbf{m}-\mathbf{e}_0-\mathbf{e}_1\in \mathbb{Z}_{\geqslant 0}^{2n}}\sum_{j\in\mathbb{Z}}{C}_{g,\mathbf{m}-\mathbf{e}_0}^j z_1^{0'}P^{\mathbf{m}-\mathbf{e}_0}_{s,j}.
\end{align*}
From Lemma \ref{lem:L action on P sj} and equation \eqref{eq:proof of thm main:case x in X:exactness:cof 2}, we obtain equations
\begin{align*}
{C}_{f,\mathbf{m}-\mathbf{e}_1}^j z_0^{1'}P^{\mathbf{m}-\mathbf{e}_1}_{s,j}+ {C}_{g,\mathbf{m}-\mathbf{e}_0}^j z_1^{1'}P^{\mathbf{m}-\mathbf{e}_0}_{s,j}
&=-{C}_{f,\mathbf{m}-\mathbf{e}_1}^jz_0^{1'}\big(z_1^{0'}P^{\mathbf{m}-\mathbf{e}_0-\mathbf{e}_1}_{s,j}+z_1^{1'}P^{\mathbf{m}-\mathbf{e}_0-\mathbf{e}_1}_{s+1,j}\big)\\
&\quad +{C}_{f,\mathbf{m}-\mathbf{e}_1}^jz_1^{1'}\big(z_0^{0'}P^{\mathbf{m}-\mathbf{e}_0-\mathbf{e}_1}_{s,j}+z_0^{1'}P^{\mathbf{m}-\mathbf{e}_0-\mathbf{e}_1}_{s+1,j}\big)\\
&={C}_{f,\mathbf{m}-\mathbf{e}_1}^j|q_0|^2P^{\mathbf{m}-\mathbf{e}_0-\mathbf{e}_1}_{s,j},
\end{align*}
and
\begin{align*}
-{C}_{f,\mathbf{m}-\mathbf{e}_1}^j z_0^{0'}P^{\mathbf{m}-\mathbf{e}_1}_{s,j}- {C}_{g,\mathbf{m}-\mathbf{e}_0}^j z_1^{0'}P^{\mathbf{m}-\mathbf{e}_0}_{s,j}
&={C}_{f,\mathbf{m}-\mathbf{e}_1}^jz_0^{0'}\big(z_1^{0'}P^{\mathbf{m}-\mathbf{e}_0-\mathbf{e}_1}_{s,j}+z_1^{1'}P^{\mathbf{m}-\mathbf{e}_0-\mathbf{e}_1}_{s+1,j}\big)\\
&\quad -{C}_{f,\mathbf{m}-\mathbf{e}_1}^jz_1^{0'}\big(z_0^{0'}P^{\mathbf{m}-\mathbf{e}_0-\mathbf{e}_1}_{s,j}+z_0^{1'}P^{\mathbf{m}-\mathbf{e}_0-\mathbf{e}_1}_{s+1,j}\big)\\
&={C}_{f,\mathbf{m}-\mathbf{e}_1}^j|q_0|^2P^{\mathbf{m}-\mathbf{e}_0-\mathbf{e}_1}_{s+1,j}
\end{align*}
which hold for every multi-index $\mathbf{m}$ such that $\mathbf{m}-\mathbf{e}_0,\mathbf{m}-\mathbf{e}_1\in \mathbb{Z}^{2n}_{\geqslant 0}$. This implies that
\begin{equation}
\begin{aligned}
     z_0^{1'} \eta(f)_s + z_1^{1'}\eta(g)_s&=|q_0|^2\sum_{\mathbf{m}-\mathbf{e}_0-\mathbf{e}_1\in \mathbb{Z}_{\geqslant 0}^{2n}}\sum_{j\in\mathbb{Z}}C_{f,\mathbf{m}-\mathbf{e}_1}^jP^{\mathbf{m}-\mathbf{e}_0-\mathbf{e}_1}_{s,j},\\
     -z_0^{0'}\eta(f)_s -z_1^{0'}\eta(g)_s&=|q_0|^2\sum_{\mathbf{m}-\mathbf{e}_0-\mathbf{e}_1\in \mathbb{Z}_{\geqslant 0}^{2n}}\sum_{j\in\mathbb{Z}}C_{f,\mathbf{m}-\mathbf{e}_1}^jP^{\mathbf{m}-\mathbf{e}_0-\mathbf{e}_1}_{s+1,j}.
\end{aligned}
\end{equation}
If the above series converges on $B_0(\varepsilon^2)$, then $h$ is well-defined and belongs to $C^1$. Let
\[
M=\sup_{x\in B(\varepsilon),\,0\leqslant s\leqslant k+1}\big\{|\eta(f)_s(x)|,\,|\eta(g)_s(x)|\big\}.
\]
Then Lemma \ref{lem:estimate of C mj and P m sj} implies the following estimate on $B_0(\varepsilon^2)$:
\begin{align*}
\sum_{\mathbf{m}-\mathbf{e}_0-\mathbf{e}_1\in \mathbb{Z}_{\geqslant 0}^{2n}}\sum_{j\in\mathbb{Z}}\big|C_{f,\mathbf{m}-\mathbf{e}_1}^j\big|\,\big|P^{\mathbf{m}-\mathbf{e}_0-\mathbf{e}_1}_{s,j}\big|
&\leqslant CM\sum_{|\mathbf{m}|=1}^\infty \big(|\mathbf{m}|+s\big) \left(\frac{2^{4n+5}n^2e}{\varepsilon}\right)^{|\mathbf{m}|}\big(8n\varepsilon^2\big)^{|\mathbf{m}|},\\[4pt]
\sum_{\mathbf{m}-\mathbf{e}_0-\mathbf{e}_1\in \mathbb{Z}_{\geqslant 0}^{2n}}\sum_{j\in\mathbb{Z}}\big|C_{f,\mathbf{m}-\mathbf{e}_1}^j\big|\,\big|P^{\mathbf{m}-\mathbf{e}_0-\mathbf{e}_1}_{s+1,j}\big|
&\leqslant CM\sum_{|\mathbf{m}|=1}^\infty \big(|\mathbf{m}|+s\big) \left(\frac{2^{4n+5}n^2e}{\varepsilon}\right)^{|\mathbf{m}|}\big(8n\varepsilon^2\big)^{|\mathbf{m}|}.
\end{align*}
Since the right-hand sides of the above inequalities converge for sufficiently small $\varepsilon$, the two series are convergent.
This implies that $\eta(h)$ is well-defined and given by
\begin{equation*}
\begin{aligned}
\eta(h)_s&=\sum_{\mathbf{m}-\mathbf{e}_0-\mathbf{e}_1\in \mathbb{Z}_{\geqslant 0}^{2n}}\sum_{j\in\mathbb{Z}}C_{f,\mathbf{m}-\mathbf{e}_1}^jP^{\mathbf{m}-\mathbf{e}_0-\mathbf{e}_1}_{s,j},\\
\eta(h)_{s+1}&=\sum_{\mathbf{m}-\mathbf{e}_0-\mathbf{e}_1\in \mathbb{Z}_{\geqslant 0}^{2n}}\sum_{j\in\mathbb{Z}}C_{f,\mathbf{m}-\mathbf{e}_1}^jP^{\mathbf{m}-\mathbf{e}_0-\mathbf{e}_1}_{s+1,j}.
\end{aligned}
\end{equation*}

\textbf{Step 3: The surjectivity of $\mathscr{L}^{(k)}_x$.}

Finally, we prove that if $h\in \mathcal{R}^{(k)}(B_0(\varepsilon))$ and $h$ vanishes on $X$, then there exist $f,g\in \mathcal{R}^{(k+1)}(B_0(\varepsilon^2))$ satisfying $L_0^{(k+1)}f+L_1^{(k+1)}g=h$.

Assume that $\eta(h)$ admits the following series expansion:
\begin{equation}\label{eq:proof of thm main:case x in X:series of h}
\eta(h)_s=\sum_{\mathbf{m}\in \mathbb{Z}_{\geqslant 0}^{2n}}\sum_{j\in\mathbb{Z}}C_{\mathbf{m}}^j P^{\mathbf{m}}_{s,j}.
\end{equation}
Since the restriction $h|_X$ vanishes identically, appealing to the Taylor expansion of $h$ given in \eqref{eq:C mj and Taylor expansion}, we deduce that $C_{\mathbf{m}}^j = 0$ provided $m_0 = m_1 = 0$.

We now consider the following two series:
\begin{align*}
\sum_{\mathbf{m}\in \mathbb{Z}_{\geqslant 0}^{2n},\,m_0=0}\sum_{j\in\mathbb{Z}}C_{\mathbf{m}}^j P^{\mathbf{m}}_{s,j},\qquad \sum_{\mathbf{m}\in \mathbb{Z}_{\geqslant 0}^{2n},\,m_1=0}\sum_{j\in\mathbb{Z}}C_{\mathbf{m}}^j P^{\mathbf{m}}_{s,j}.
\end{align*}
By Lemma \ref{lem:estimate of C mj and P m sj}, the expansion of $\eta(h)_s$ presented in \eqref{eq:proof of thm main:case x in X:series of h} converges absolutely. Therefore, both series converge on $B_0(\varepsilon)$ and satisfy
\[
\sum_{\mathbf{m}\in \mathbb{Z}_{\geqslant 0}^{2n},\,m_0=0}\sum_{j\in\mathbb{Z}}C_{\mathbf{m}}^j P^{\mathbf{m}}_{s,j}
+
\sum_{\mathbf{m}\in \mathbb{Z}_{\geqslant 0}^{2n},\,m_1=0}\sum_{j\in\mathbb{Z}}C_{\mathbf{m}}^j P^{\mathbf{m}}_{s,j}
=\eta(h)_s.
\]
From Lemma \ref{lem:L action on P sj}, it follows that
\begin{equation*}
\begin{aligned}
\sum_{\mathbf{m}\in \mathbb{Z}_{\geqslant 0}^{2n},\,m_0=0}\sum_{j\in\mathbb{Z}}C_{\mathbf{m}}^j P^{\mathbf{m}}_{s,j}
&=z_{0}^{0'}\sum_{\mathbf{m}\in \mathbb{Z}_{\geqslant 0}^{2n},\,m_0=0}\sum_{j\in\mathbb{Z}}C_{\mathbf{m}}^j P^{\mathbf{m}-\mathbf{e}_1}_{s,j}\\
&\quad +z_{0}^{1'}\sum_{\mathbf{m}\in \mathbb{Z}_{\geqslant 0}^{2n},\,m_0=0}\sum_{j\in\mathbb{Z}}C_{\mathbf{m}}^j P^{\mathbf{m}-\mathbf{e}_1}_{s+1,j},
\end{aligned}
\end{equation*}
and
\begin{equation*}
\begin{aligned}
\sum_{\mathbf{m}\in \mathbb{Z}_{\geqslant 0}^{2n},\,m_1=0}\sum_{j\in\mathbb{Z}}C_{\mathbf{m}}^j P^{\mathbf{m}}_{s,j}
&=-z_{1}^{0'}\sum_{\mathbf{m}\in \mathbb{Z}_{\geqslant 0}^{2n},\,m_1=0}\sum_{j\in\mathbb{Z}}C_{\mathbf{m}}^j P^{\mathbf{m}-\mathbf{e}_0}_{s,j}\\
&\quad -z_{1}^{1'}\sum_{\mathbf{m}\in \mathbb{Z}_{\geqslant 0}^{2n},\,m_1=0}\sum_{j\in\mathbb{Z}}C_{\mathbf{m}}^j P^{\mathbf{m}-\mathbf{e}_0}_{s+1,j}.
\end{aligned}
\end{equation*}

The above series converge on $B_0(\varepsilon^2)$, owing to the following estimates:
\begin{equation*}
\sum_{\mathbf{m}\in \mathbb{Z}_{\geqslant 0}^{2n},\,m_0=0}\sum_{j\in\mathbb{Z}}\big|C_{\mathbf{m}}^j\big| \big|P^{\mathbf{m}-\mathbf{e}_1}_{s,j}\big|
\leqslant CM\sum_{|\mathbf{m}|=1}^{\infty}\big(|\mathbf{m}|+s+1\big)\left(\frac{2^{4n+5}n^2e}{\varepsilon}\right)^{|\mathbf{m}|}\big(8n\varepsilon^2\big)^{|\mathbf{m}|-1},
\end{equation*}
\begin{equation*}
\sum_{\mathbf{m}\in \mathbb{Z}_{\geqslant 0}^{2n},\,m_1=0}\sum_{j\in\mathbb{Z}}\big|C_{\mathbf{m}}^j\big| \big|P^{\mathbf{m}-\mathbf{e}_0}_{s,j}\big|
\leqslant CM\sum_{|\mathbf{m}|=1}^{\infty}\big(|\mathbf{m}|+s+1\big)\left(\frac{2^{4n+5}n^2e}{\varepsilon}\right)^{|\mathbf{m}|}\big(8n\varepsilon^2\big)^{|\mathbf{m}|-1}.
\end{equation*}

Here $M$ is given by
\[
M=\sup_{x\in B(\varepsilon),\,0\leqslant s\leqslant k}\big\{|\eta(h)_s(x)|\big\}.
\]

Then we define $\eta(f)$ and $\eta(g)$ by
\begin{equation*}
\begin{aligned}
\eta(f)_s&:= \sum_{\mathbf{m}\in \mathbb{Z}_{\geqslant 0}^{2n},\,m_0=0}\sum_{j\in\mathbb{Z}}C_{\mathbf{m}}^j P^{\mathbf{m}-\mathbf{e}_1}_{s,j},\\
\eta(g)_s&:= -\sum_{\mathbf{m}\in \mathbb{Z}_{\geqslant 0}^{2n},\,m_1=0}\sum_{j\in\mathbb{Z}}C_{\mathbf{m}}^j P^{\mathbf{m}-\mathbf{e}_0}_{s,j}.
\end{aligned}
\end{equation*}
Consequently, $f$ and $g$ satisfy
\begin{equation*}
L_0^{(k+1)}f+L_1^{(k+1)}g=h.
\end{equation*}
This completes the proof. \qed

\section{A long exact sequence of differential forms}\label{section:long exact sequence}

From the perspective of homological algebra, the Koszul complex \eqref{eq:Koszul complex} gives rise to a long exact sequence of sheaf cohomology groups. Nevertheless, this sequence is purely abstract and thus inconvenient for practical applications. In this section, we establish a long exact sequence of cohomology groups expressed in terms of differential forms, and complete the proof of Theorem \ref{thm:long exact sequence}.

We define operators between differential forms as follows.
\begin{definition}\label{def: L and tilde L}
For integers $k$ and $j$, any element in $C^\infty(\Omega,V_j^{(k)})$ is denoted by $(f_{\AAAAA})$.
Define operators $\mathcal{L}^{(k)}_j:C^\infty(\Omega,V_j^{(k+1)}\oplus V_j^{(k+1)})\longrightarrow C^\infty(\Omega,V_{j}^{(k)})$ by
    \begin{equation*}
       \left(\mathcal{L}^{(k)}_j(f,g)^T\right)_{\AAAAA }=\begin{cases}
          z_0^{0^\prime}f_{0^\prime \AAAAA }+z_0^{1^\prime}f_{1^\prime \AAAAA }
+ z_1^{0^\prime}g_{0^\prime \AAAAA }+z_1^{1^\prime}g_{1^\prime \AAAAA },
& 0\leqslant j\leqslant k,
\\[4pt]
z_0^{1^\prime}d^{0^\prime}f-z_0^{0^\prime}d^{1^\prime}f-z_1^{0^\prime}d^{1^\prime}g+z_1^{1^\prime}d^{0^\prime}g+2\omega^1\wedge f-2\omega^0\wedge g,
& j = k+1,
\\[4pt]
-z_0^{(A_0^\prime}f^{A_1^\prime \cdots A_{j-k-2}^\prime)}-z_1^{(A_0^\prime}g^{A_1^\prime \cdots A_{j-k-2}^\prime)},
& j\geqslant k+2.
        \end{cases}
    \end{equation*}
Define operators $\widetilde{\mathcal{L}}^{(k)}_j:C^\infty(\Omega,V_j^{(k+2)})\longrightarrow C^\infty(\Omega,V_j^{(k+1)}\oplus V_j^{(k+1)})$ by
\begin{equation*}
      \left(\widetilde{\mathcal{L}}^{(k)}_j(h)\right)_{\AAAAA }=\begin{cases}
          \left(-z_1^{0^\prime}h_{0^\prime \AAAAA }-z_1^{1^\prime}h_{1^\prime \AAAAA },
 z_0^{0^\prime}h_{0^\prime \AAAAA }+z_0^{1^\prime}h_{1^\prime \AAAAA }\right)^T,
& 0\leqslant j\leqslant k+1 ,
\\[4pt]
\begin{pmatrix}2\omega^0\wedge h+z_1^{0^\prime}d^{1^\prime}h-z_1^{1^\prime}d^{0^\prime}h \\
2\omega^1\wedge h-z_0^{0^\prime}d^{1^\prime}h+z_0^{1^\prime}d^{0^\prime}h\end{pmatrix},
& j = k+2,
\\[4pt]
\left(z_1^{(A_0^\prime}h^{A_1^\prime \cdots A_{j-k-3}^\prime)},-z_0^{(A_0^\prime}h^{A_1^\prime \cdots A_{j-k-3}^\prime)}\right)^T,
& j\geqslant k+3.
\end{cases}
\end{equation*}
\end{definition}
More explicitly,  the operators $\mathcal{L}_j^{(k)}$ can be expressed as
\begin{equation}\label{eq:LA and mathcal L}
\mathcal{L}_j^{(k)}\bigl(f,g\bigr)^T=L_0^{(k-j+1)}f+L_1^{(k-j+1)}g.
\end{equation}
when $j\leqslant k$. Meanwhile, the operators $\widetilde{\mathcal{L}}_j^{(k)}$ take the form
\begin{equation}\label{eq:LA and widetilde mathcal L}
\widetilde{\mathcal{L}}_j^{(k)}h=\bigl(L_1^{(k-j+2)}h,\,L_0^{(k-j+2)}h\bigr)^T.
\end{equation}
for $j\leqslant k+1$.

Using the isomorphism $\eta$ \eqref{eq:def:eta} and Lemma \ref{lem:sym}, we can simplify the expressions of the operators $\mathcal{L}_j^{(k)}$ as follows:
\begin{equation}\label{eq:def mathcal L eta}
\eta\bigl(\mathcal{L}^{(k)}_j\bigl(f,g\bigr)^T\bigr)_{s}=\begin{cases}
          z_0^{0^\prime}\eta(f)_{s}+z_0^{1^\prime}\eta(f)_{s+1}
+ z_1^{0^\prime}\eta(g)_{s}+z_1^{1^\prime}\eta(g)_{s+1},
& 0\leqslant j\leqslant k,
\\[4pt]
 -\frac{p-s}{p}(z_0^{0'}\eta(f)_{s}+z_1^{0'}\eta(g)_{s})-\frac{s}{p}(z_0^{1'}\eta(f)_{s-1}+z_1^{1'}\eta(g)_{s-1}),
& j\geqslant k+2.
\end{cases}
\end{equation}
Here we set
\[
p = p(k,j) + 1,\quad s = 0,\dots,p.
\]
For the definition of $p(k,j)$, see equation \eqref{eq:def:p kj}.
Similarly, we simplify the expressions of the operators $\widetilde{\mathcal{L}}_j^{(k)}$ as follows:
\begin{equation}\label{eq:def widetilde mathcal L eta}
\eta\left(\widetilde{\mathcal{L}}^{(k)}_j(h)\right)_{s}=\begin{cases}
 \left(-z_1^{0^\prime}\eta(h)_{s}-z_1^{1^\prime}\eta(h)_{s+1},
 z_0^{0^\prime}\eta(h)_{s}+z_0^{1^\prime}\eta(h)_{s+1}\right)^T,
& 0\leqslant j\leqslant k+1 ,
\\[4pt]
\begin{pmatrix}
  \frac{p-s}{p}z_1^{0'}\eta(h)_s+\frac{s}{p}z_1^{1'}\eta(h)_{s-1} \\
  -\frac{p-s}{p}z_0^{0'}\eta(h)_s-\frac{s}{p}z_0^{1'}\eta(h)_{s-1}
\end{pmatrix},
& j\geqslant k+3.
\end{cases}
\end{equation}
Here we set
\[
p=p(k+1,j)+1,\qquad s=0,\ldots,p.
\]

We stress that the images of the operators $\mathcal{L}_{k+1}^{(k)}$ and $\widetilde{\mathcal{L}}_{k+2}^{(k)}$ are devoid of symmetric tensor components on $\mathbb{C}^2$. Consequently, the isomorphism $\eta$ does not yield any simplification to the representations of $\mathcal{L}_{k+1}^{(k)}$ and $\widetilde{\mathcal{L}}_{k+2}^{(k)}$.

We establish several properties associated with the operators $\mathcal{L}_j^{(k)}$, $\widetilde{\mathcal{L}}_j^{(k)}$, and $\mathscr{D}_j^{(k)}$.

\begin{proposition}\label{prop:DL and D tilde L commutative}
 The identities
    \[
    \mathcal{L}_{j+1}^{(k)}\mathscr{D}_j^{(k+1)}=\mathscr{D}_j^{(k)}\mathcal{L}_{j}^{(k)}, \qquad     \widetilde{\mathcal{L}}_{j+1}^{(k)}\mathscr{D}_j^{(k+2)}=\mathscr{D}_j^{(k+1)}\widetilde{\mathcal{L}}_{j}^{(k)}
    \]
    hold for all $j,k$.
\end{proposition}
\begin{proof}
 The proof proceeds by cases. We begin by establishing the identity
\[
\mathcal{L}_{j+1}^{(k)}\mathscr{D}_j^{(k+1)}=\mathscr{D}_j^{(k)}\mathcal{L}_{j}^{(k)}.
\]

\textbf{Case $j<k$.}

This conclusion can be directly derived from equation \eqref{eq:LA and mathcal L}, the definition of $D^{(k)}$ in \eqref{eq:def Dk}, the definition of $\mathscr{D}_j^{(k)}$ in \eqref{eq:def_of_Dj}, and Proposition~\ref{prop:Lj and k CF}.

\textbf{Case $j=k$.}

We use the fact that $\mathscr{D}_{k}^{(k)}=d^{0'}d^{1'}$. Then Lemma~\ref{lem:computation z} implies that
\begin{align*}
d^{0'}d^{1'}\bigl(\mathcal{L}_{k}^{(k)}\bigl(f,g\bigr)^T\bigr)
&= d^{0'}d^{1'}\bigl(z_0^{0'}f_{0'}+z_0^{1'}f_{1'}
+ z_1^{0'}g_{0'}+z_1^{1'}g_{1'}\bigr)\\
&=z_0^{0'}d^{0'}d^{1'}f_{0'}
+z_0^{1'}d^{0'}d^{1'}f_{1'}
+ z_1^{0'}d^{0'}d^{1'}g_{0'}
+z_1^{1'}d^{0'}d^{1'}g_{1'}\\
&\quad +2\omega^1 \wedge d^{0'}f_{0'}
+2\omega^1\wedge d^{1'}f_{1'}
-2\omega^0\wedge d^{0'}g_{0'}
-2\omega^0\wedge d^{1'}g_{1'}.
\end{align*}
On the other hand, by the definition of $\mathcal{L}_k^{(k+1)}$, we have
\begin{align*}
\mathcal{L}_{k+1}^{(k)}\bigl(\mathscr{D}_k^{(k+1)}f,\mathscr{D}_k^{(k+1)}g\bigr)^T
&=-z_0^{0'}d^{1'}\mathscr{D}_k^{(k+1)}f
+z_0^{1'}d^{0'}\mathscr{D}_k^{(k+1)}f\\
&\quad -z_1^{0'}d^{1'}\mathscr{D}_k^{(k+1)}g
+z_1^{1'}d^{0'}\mathscr{D}_k^{(k+1)}g\\
&\quad +2\omega^1\wedge \mathscr{D}_k^{(k+1)}f
-2\omega^0\wedge \mathscr{D}_k^{(k+1)}g.
\end{align*}
Recall the definition of $\mathscr{D}_k^{(k+1)}$ \eqref{eq:def_of_Dj} and the properties of $d^{0'}$ and $d^{1'}$ \eqref{eq:prop of d0 and d1}. We obtain
\begin{align*}
\mathcal{L}_{k+1}^{(k)}\bigl(\mathscr{D}_k^{(k+1)}f,\mathscr{D}_k^{(k+1)}g\bigr)^T
&=z_0^{0'}d^{0'}d^{1'}f_{0'}
+z_0^{1'}d^{0'}d^{1'}f_{1'}
+ z_1^{0'}d^{0'}d^{1'}g_{0'}
+z_1^{1'}d^{0'}d^{1'}g_{1'}\\
&\quad +2\omega^1 \wedge d^{0'}f_{0'}
+2\omega^1\wedge d^{1'}f_{1'}
-2\omega^0\wedge d^{0'}g_{0'}
-2\omega^0\wedge d^{1'}g_{1'}\\
&=d^{0'}d^{1'}\bigl(\mathcal{L}_{k}^{(k)}\bigl(f,g\bigr)^T\bigr).
\end{align*}

\textbf{Case $j=k+1$.}

Recall that $\mathscr{D}_{k+1}^{(k+1)}=d^{0'}d^{1'}$. We have
\begin{align*}
\left(\mathcal{L}_{k+2}^{(k)}\bigl(d^{0'}d^{1'}f,d^{0'}d^{1'}g\bigr)^T\right)^{A_0^\prime}
=-z_0^{A_0'}d^{0'}d^{1'}f-z_1^{A_0'}d^{0'}d^{1'}g.
\end{align*}
On the other hand, from the definition of $\mathscr{D}_{k+1}^{(k)}$ in \eqref{eq:def_of_Dj}
and the definition of $\mathcal{L}_{k+1}^{(k)}$ (see Definition \ref{def: L and tilde L}), we obtain
\begin{align*}
\left(\mathscr{D}_{k+1}^{(k)}\mathcal{L}_{k+1}^{(k)}\bigl(f,g\bigr)^T\right)^{A_0'}
&=d^{A_0'}\left(\mathcal{L}_{k+1}^{(k)}\bigl(f,g\bigr)^T\right)\\
&=d^{A_0'}\Bigl(
-z_0^{0^\prime}d^{1^\prime}f
+z_0^{1^\prime}d^{0^\prime}f
-z_1^{0\prime}d^{1^\prime}g
+z_1^{1^\prime}d^{0^\prime}g
+2\omega^1\wedge f
-2\omega^0\wedge g
\Bigr).
\end{align*}
Applying Lemma \ref{lem:computation z}, we get
\begin{equation}\label{eq:prop:DL and D tilde L commutative:part 1: j eq k+1: 1}
\begin{aligned}
d^{A_0'}\bigl(-z_0^{0'}d^{1'}f+z_0^{1'}d^{0'}f+2\omega^1\wedge f\bigr)
&=-\delta_{0'}^{A_0'}z_0^{0'}d^{0'}d^{1'}f
-\delta_{1'}^{A_0'}z_0^{1'}d^{0'}d^{1'}f\\
&\quad +2\delta^{A_0'}_{1^\prime}\omega^1\wedge d^{1'}f
+2\delta^{A_0'}_{0^\prime}\omega^1\wedge d^{0'}f
-2\omega^1\wedge d^{A_0'}f\\
&=-z_0^{A_0'}d^{0'}d^{1'}f
\end{aligned}
\end{equation}
and
\begin{equation}\label{eq:prop:DL and D tilde L commutative:part 1: j eq k+1: 2}
\begin{aligned}
d^{A_0'}\bigl(-z_1^{0'}d^{1'}g+z_1^{1'}d^{0'}g-2\omega^0\wedge g\bigr)
&=-\delta_{0'}^{A_0'}z_1^{0'}d^{0'}d^{1'}g
-\delta_{1'}^{A_0'}z_1^{1'}d^{0'}d^{1'}g\\
&\quad -2\delta^{A_0'}_{1^\prime}\omega^0\wedge d^{1'}g
-2\delta^{A_0'}_{0^\prime}\omega^0\wedge d^{0'}g
+2\omega^0\wedge d^{A_0'}g\\
&=-z_1^{A_0'}d^{0'}d^{1'}g.
\end{aligned}
\end{equation}
This implies that
\[\left(\mathcal{L}_{k+2}^{(k)}\bigl(d^{0'}d^{1'}f,d^{0'}d^{1'}g\bigr)^T\right)^{A_0^\prime}=\left(\mathscr{D}_{k+1}^{(k)}\mathcal{L}_{k+1}^{(k)}\bigl(f,g\bigr)^T\right)^{A_0'}.\]

\textbf{Case $j>k+1$.}

Recall that \(V_{j+1}^{(k)}\cong \odot^{j-k}\mathbb{C}^2\otimes \Lambda^{j+2}\mathbb{C}^{2n}\). Let \(p=j-k>1\). It follows from equation \eqref{eq:Djk under eta j large}  and Lemma \ref{lem:sym} that
\begin{equation}\label{eq:prop:DL and D tilde L commutative:part 1: j geq k: 1}
\begin{aligned}
\eta\left(\mathscr{D}_j^{(k)}\mathcal{L}_{j}^{(k)}\bigl(f,g\bigr)^T\right)_{s}
&=\frac{p+1-s}{p+1}d^{0'}\left(\eta\bigl(\mathcal{L}_{j}^{(k)}\bigl(f,g\bigr)^T\bigr)_{s}\right)\\
&\quad +\frac{s}{p+1}d^{1'}\left(\eta\bigl(\mathcal{L}_{j}^{(k)}\bigl(f,g\bigr)^T\bigr)_{s-1}\right).
\end{aligned}
\end{equation}
Let $F:=\eta(f)$ and $G:=\eta(g)$. From equation \eqref{eq:def widetilde mathcal L eta}, we obtain
\begin{align*}
-\eta\left(\mathcal{L}_{j}^{(k)}\bigl(f,g\bigr)^T\right)_{s}
&= \frac{p-s}{p}z_0^{0'}F_{s}+\frac{s}{p}z_0^{1'}F_{s-1}+\frac{p-s}{p}z_1^{0'}G_{s}+\frac{s}{p}z_1^{1'}G_{s-1}.
\end{align*}
Then equation \eqref{eq:prop:DL and D tilde L commutative:part 1: j geq k: 1} implies that
\begin{equation}\label{eq:prop:DL and D tilde L commutative:part 1: j geq k: 2}
    \begin{aligned}
-p(p+1)\eta\left(\mathscr{D}_j^{(k)}\mathcal{L}_{j}^{(k)}\bigl(f,g\bigr)^T\right)_{s}
&=(p+1-s)d^{0'}\left((p-s)z_0^{0'}F_{s}+sz_0^{1'}F_{s-1}\right.\\
&\quad +\left.(p-s)z_1^{0'}G_{s}+sz_1^{1'}G_{s-1}\right)\\
&\quad +sd^{1'}\left((p-s+1)z_0^{0'}F_{s-1}+(s-1)z_0^{1'}F_{s-2}\right.\\
&\quad +\left.(p-s+1)z_1^{0'}G_{s-1}+(s-1)z_1^{1'}G_{s-2}\right).
    \end{aligned}
\end{equation}
Using Lemma \ref{lem:computation z}, it follows that
\begin{equation}\label{eq:prop:DL and D tilde L commutative:part 1: j geq k: 3}
\begin{aligned}
d^{A'}\bigl((p-s)z_0^{0'}F_{s}+sz_0^{1'}F_{s-1}\bigr)
&=(p-s)z_0^{0'}d^{A'}F_{s}+sz_0^{1'}d^{A'}F_{s-1}\\
&\quad +2(p-s)\epsilon^{0'A'}\omega^1\wedge F_s-2s\epsilon^{1'A'}\omega^1\wedge F_{s-1},
\end{aligned}
\end{equation}
and
\begin{equation}\label{eq:prop:DL and D tilde L commutative:part 1: j geq k: 4}
\begin{aligned}
d^{A'}\bigl((p-s)z_1^{0'}G_{s}+sz_1^{1'}G_{s-1}\bigr)
&=(p-s)z_1^{0'}d^{A'}G_{s}+sz_1^{1'}d^{A'}G_{s-1}\\
&\quad -2(p-s)\epsilon^{0'A'}\omega^0\wedge G_s+2s\epsilon^{1'A'}\omega^0\wedge G_{s-1}.
\end{aligned}
\end{equation}
Then equations \eqref{eq:prop:DL and D tilde L commutative:part 1: j geq k: 2},
\eqref{eq:prop:DL and D tilde L commutative:part 1: j geq k: 3} and
\eqref{eq:prop:DL and D tilde L commutative:part 1: j geq k: 4} imply that
\begin{equation}\label{eq:prop:DL and D tilde L commutative:part 1: j geq k: 5}
\begin{aligned}
-p(p+1)\eta\left(\mathscr{D}_j^{(k)}\mathcal{L}_{j}^{(k)}\bigl(f,g\bigr)^T\right)_{s}
&=(p+1-s)(p-s)(z_0^{0'}d^{0'}F_s+z_1^{0'}d^{0'}G_s)\\
&\quad +(p+1-s)s(z_0^{1'}d^{0'}F_{s-1}+z_1^{1'}d^{0'}G_{s-1})\\
&\quad +(p+1-s)s(z_0^{0'}d^{1'}F_{s-1}+z_1^{0'}d^{1'}G_{s-1})\\
&\quad +s(s-1)(z_0^{1'}d^{1'}F_{s-2}+z_1^{1'}d^{1'}G_{s-2}).
\end{aligned}
\end{equation}

On the other hand, by virtue of Lemma \ref{lem:sym} together with equations \eqref{eq:Djk under eta j large}, \eqref{eq:prop:DL and D tilde L commutative:part 1: j geq k: 1} and \eqref{eq:prop:DL and D tilde L commutative:part 1: j geq k: 5}, we deduce that
\begin{equation}\label{eq:prop:DL and D tilde L commutative:part 1: j geq k: final}
\begin{aligned}
p(p+1)\eta\left(\mathcal{L}_{j+1}^{(k)}\mathscr{D}_j^{(k+1)}\bigl(f,g\bigr)^T\right)_s
&=p(p+1-s)z_0^{0'}\eta\big(\mathscr{D}_j^{(k+1)}f\big)_s\\
&\quad +psz_0^{1'}\eta\big(\mathscr{D}_j^{(k+1)}f\big)_{s-1}\\
&\quad +p(p+1-s)z_1^{0'}\eta\big(\mathscr{D}_j^{(k+1)}g\big)_s\\
&\quad +psz_1^{1'}\eta\big(\mathscr{D}_j^{(k+1)}g\big)_{s-1}\\
&=p(p+1)\eta\left(\mathscr{D}_j^{(k)}\mathcal{L}_{j}^{(k)}\bigl(f,g\bigr)^T\right)_{s}.
\end{aligned}
\end{equation}
This verifies the case $j>k+1$ and thus finishes the proof of the identity
\[
\mathcal{L}_{j+1}^{(k)}\mathscr{D}_j^{(k+1)}=\mathscr{D}_j^{(k)}\mathcal{L}_{j}^{(k)}.
\]

The assertion
\[\widetilde{\mathcal{L}}_{j+1}^{(k)}\mathscr{D}_j^{(k+2)}=\mathscr{D}_j^{(k+1)}\widetilde{\mathcal{L}}_{j}^{(k)}\]
can be proved similarly. The entire proof is divided into several cases.

\textbf{Case $j<k+1$.}

This conclusion follows directly from equation \eqref{eq:LA and widetilde mathcal L}, the definition of $D^{(k)}$ given in \eqref{eq:def Dk}, the definition of $\mathscr{D}_j^{(k)}$ in \eqref{eq:def_of_Dj}, and Proposition~\ref{prop:Lj and k CF}.

\textbf{Case $j=k+1$.}

Note that $\mathscr{D}_{k+1}^{(k+1)}=d^{0'}d^{1'}$. Then we have
\[
d^{0'}d^{1'}\bigl(\widetilde{\mathcal{L}}^{(k)}_{k+1}h\bigr)
=
\begin{pmatrix}
d^{0'}d^{1'}\bigl(-z_1^{1'}h_{1'}-z_1^{0'}h_{0'}\bigr)\\[4pt]
d^{0'}d^{1'}\bigl(z_0^{1'}h_{1'}+z_0^{0'}h_{0'}\bigr)
\end{pmatrix}.
\]
By Lemma~\ref{lem:computation z}, it follows that
\[
d^{0'}d^{1'}\bigl(\widetilde{\mathcal{L}}^{(k)}_{k+1}h\bigr)
=
\begin{pmatrix}
-z_1^{1'}d^{0'}d^{1'}h_{1'}-z_1^{0'}d^{0'}d^{1'}h_{0'}+2\omega^0\wedge\bigl(d^{1'}h_{1'}+d^{0'}h_{0'}\bigr)\\[4pt]
z_0^{1'}d^{0'}d^{1'}h_{1'}+z_0^{0'}d^{0'}d^{1'}h_{0'}+2\omega^1\wedge\bigl(d^{1'}h_{1'}+d^{0'}h_{0'}\bigr)
\end{pmatrix}.
\]
On the other hand, from the definition of $\widetilde{\mathcal{L}}_{k+2}^{(k)}$ (see Definition~\ref{def: L and tilde L}), we obtain
\begin{align*}
&\widetilde{\mathcal{L}}_{k+2}^{(k)}\bigl(\mathscr{D}_{k+1}^{(k+2)}h\bigr)=
\widetilde{\mathcal{L}}_{k+2}^{(k)}\bigl(d^{0'}h_{0'}+d^{1'}h_{1'}\bigr)\\
&=
\begin{pmatrix}
2\omega^0\wedge\bigl(d^{0'}h_{0'}+d^{1'}h_{1'}\bigr)
+z_1^{0'}d^{1'}\bigl(d^{0'}h_{0'}+d^{1'}h_{1'}\bigr)
-z_1^{1'}d^{0'}\bigl(d^{0'}h_{0'}+d^{1'}h_{1'}\bigr)\\[4pt]
2\omega^1\wedge\bigl(d^{0'}h_{0'}+d^{1'}h_{1'}\bigr)
-z_0^{0'}d^{1'}\bigl(d^{0'}h_{0'}+d^{1'}h_{1'}\bigr)
+z_0^{1'}d^{0'}\bigl(d^{0'}h_{0'}+d^{1'}h_{1'}\bigr)
\end{pmatrix}\\
&=
\begin{pmatrix}
2\omega^0\wedge\bigl(d^{0'}h_{0'}+d^{1'}h_{1'}\bigr)
+z_1^{0'}d^{1'}d^{0'}h_{0'}
-z_1^{1'}d^{0'}d^{1'}h_{1'}\\[4pt]
2\omega^1\wedge\bigl(d^{0'}h_{0'}+d^{1'}h_{1'}\bigr)
-z_0^{0'}d^{1'}d^{0'}h_{0'}
+z_0^{1'}d^{0'}d^{1'}h_{1'}
\end{pmatrix}=d^{0'}d^{1'}\bigl(\widetilde{\mathcal{L}}^{(k)}_{k+1}h\bigr).
\end{align*}

\textbf{Case $j=k+2$.}

Note that $\mathscr{D}_{k+2}^{(k+2)}=d^{0'}d^{1'}$. Then we have
\begin{align*}
\bigl(\widetilde{\mathcal L}^{(k)}_{k+3}\mathscr{D}_{k+2}^{(k+2)}h\bigr)^{A_0'}
&= \bigl(z_1^{A_0'}d^{0'}d^{1'}h,\, -z_0^{A_0'}d^{0'}d^{1'}h\bigr)^T.
\end{align*}
On the other hand, by setting $f=-g=h$ in equations \eqref{eq:prop:DL and D tilde L commutative:part 1: j eq k+1: 1} and \eqref{eq:prop:DL and D tilde L commutative:part 1: j eq k+1: 2}, we obtain
\begin{align*}
\bigl(\mathscr{D}_{k+2}^{(k+1)}\widetilde{\mathcal L}^{(k)}_{k+2}h\bigr)^{A_0'}
&= \begin{pmatrix}
d^{A_0'}\bigl(-z_1^{1'}d^{0'}h + z_1^{0'}d^{1'}h + 2\omega^0\wedge h\bigr)\\[4pt]
d^{A_0'}\bigl(\phantom{-}z_0^{1'}d^{0'}h - z_0^{0'}d^{1'}h + 2\omega^1\wedge h\bigr)
\end{pmatrix}\\[4pt]
&= \bigl(z_1^{A_0'}d^{0'}d^{1'}h,\, -z_0^{A_0'}d^{0'}d^{1'}h\bigr)^T
= \bigl(\widetilde{\mathcal L}^{(k)}_{k+3}\mathscr{D}_{k+2}^{(k+2)}h\bigr)^{A_0'}.
\end{align*}

\textbf{Case $j>k+2$.}

Recall that \(V_{j+1}^{(k+1)}\cong \odot^{j-k-1}\mathbb{C}^2\otimes \Lambda^{j+2}\mathbb{C}^{2n}\).
Let \(p=j-k-1>1\).
It follows from equation \eqref{eq:Djk under eta j large}  and Lemma \ref{lem:sym} that
\begin{equation}\label{eq:prop:DL and D tilde L commutative:part 2: j geq k: 1}
  \eta\bigl(\mathscr{D}_j^{(k+1)}\widetilde{\mathcal{L}}_j^{(k)}h\bigr)_s =\frac{p+1-s}{p+1}d^{0'}\, \eta(\widetilde{\mathcal{L}}_j^{(k)}h)_s+\frac{s}{p+1}d^{1'}\, \eta(\widetilde{\mathcal{L}}_j^{(k)}h)_{s-1}.
\end{equation}
Let $H:=\eta(h)$. From equation \eqref{eq:def widetilde mathcal L eta}, we obtain
\begin{equation*}
     \eta\bigl(\widetilde{\mathcal{L}}_j^{(k)}h\bigr)_s=\begin{pmatrix}
  \frac{p-s}{p}z_1^{0'}H_s+\frac{s}{p}z_1^{1'}H_{s-1} \\
  -\frac{p-s}{p}z_0^{0'}H_s-\frac{s}{p}z_0^{1'}H_{s-1}
\end{pmatrix}.
\end{equation*}
By equation \eqref{eq:prop:DL and D tilde L commutative:part 2: j geq k: 1}, we have
\begin{equation*}
    \begin{aligned}
        &p(p+1)\eta\bigl(\mathscr{D}_j^{(k+1)}\widetilde{\mathcal{L}}_j^{(k)}h\bigr)_s=\bigl(\sigma_1,-\sigma_2\bigr)^T,
        \end{aligned}
        \end{equation*}
where $\sigma_1$ and $\sigma_2$ are given as following
\begin{equation*}
\begin{aligned}
        \sigma_1&=
           (p+1-s)d^{0'}((p-s)z_1^{0'}H_s+sz_1^{1'}H_{s-1})+sd^{1'}((p-s+1)z_1^{0'}H_{s-1}+(s-1)z_1^{1'}H_{s-2}),\\
          \sigma_2&=(p+1-s)d^{0'}((p-s)z_0^{0'}H_s+sz_0^{1'}H_{s-1})+sd^{1'}((p-s+1)z_0^{0'}H_{s-1}+(s-1)z_0^{1'}H_{s-2}).
    \end{aligned}
\end{equation*}
Set $F = H$ in \eqref{eq:prop:DL and D tilde L commutative:part 1: j geq k: 3} and $G = -H$ in \eqref{eq:prop:DL and D tilde L commutative:part 1: j geq k: 4}. Combining this with Lemma \ref{lem:computation z}, we obtain
\begin{equation*}
    \begin{aligned}
    \sigma_1&=(p+1-s)\big((p-s)z_1^{0'}d^{0'}H_s+sz_1^{1'}d^{0'}H_{s-1}\big)\\
    &\quad +s\big((p-s+1)z_1^{0'}d^{1'}H_{s-1}+(s-1)z_1^{1'}d^{1'}H_{s-2}\big),\\
    \sigma_2&=(p+1-s)\big((p-s)z_0^{0'}d^{0'}H_s+sz_0^{1'}d^{0'}H_{s-1}\big)\\
    &\quad +s\big((p-s+1)z_0^{0'}d^{1'}H_{s-1}+(s-1)z_0^{1'}d^{1'}H_{s-2}\big).
    \end{aligned}
\end{equation*}

The desired result then follows by expanding $p(p+1)\eta\left(\widetilde{\mathcal{L}}_{j+1}^{(k)}\mathscr{D}_{j}^{(k+2)}h\right)_s$, whose form is analogous to \eqref{eq:prop:DL and D tilde L commutative:part 1: j geq k: final}.

\end{proof}

\begin{proposition}\label{prop: injective of tilde L}
 Let $\Omega\subset\mathbb{H}^n$ be a domain. The operator $\widetilde{\mathcal{L}}_j^{(k)}$ is injective for all $j\neq k+2$.
\end{proposition}

\begin{proof}
   Recall the definition of $\widetilde{\mathcal{L}}_j^{(k)}$ (see Definition \ref{def: L and tilde L}). The proof is divided into two cases.

\textbf{Case $0\leqslant j\leqslant k+1$.}

The argument follows from the proof of the injectivity of $\widetilde{\mathscr L}^{(k+2-j)}$ in the case $x\notin X$ given in Theorem \ref{thm:main}.

\textbf{Case $j > k+2$.}

Recall that the image of $\widetilde{\mathcal L}_j^{(k)}$ is contained in $C^\infty(\Omega,V_j^{(k+1)}\oplus V_j^{(k+1)})$, where \(V_{j}^{(k+1)} = \odot^{j-k-2}\mathbb{C}^2 \otimes \Lambda^{j+1}\mathbb{C}^{2n}\). Let \(p = j - k -2 > 0\).
Recall equation \eqref{eq:def widetilde mathcal L eta}. We obtain
\begin{align*}
\eta\big(\widetilde{\mathcal{L}}_j^{(k)}h\big)_{s}
&= \begin{pmatrix}
\frac{s}{p} z_1^{1'} (\eta (h))_{s-1} + \frac{p-s}{p} z_1^{0'} (\eta (h))_{s} \\[6pt]
-\frac{s}{p} z_0^{1'} (\eta (h))_{s-1} - \frac{p-s}{p} z_0^{0'} (\eta (h))_{s}
\end{pmatrix}.
\end{align*}

Assume that $\widetilde{\mathcal L}_j^{(k)}h=0$. It is clear that $\eta(\widetilde{\mathcal L}_j^{(k)}h)=0$. Then for any fixed $s=1,\ldots,p-1$, $\eta(h)_{s-1}=\eta(h)_s=0$ since the determinant of matrix
\[\begin{pmatrix}\frac{s}{p}  z_1^{1'} & \frac{p-s}{p}z_{1}^{0'}   \\
-\frac{s}{p}  z_0^{1'} &-\frac{p-s}{p}z_{0}^{0'}\end{pmatrix}\]
does not vanish on \(\Omega \cap \{q_0 \neq 0\}\).

Since \(\eta\) is an isomorphism, we conclude that \(h = 0\).
\end{proof}

\begin{remark}
 In general, the operator $\widetilde{\mathcal{L}}_{k+2}^{(k)}$ fails to be injective. Indeed, one has
\[
\widetilde{\mathcal{L}}_{2}^{(0)}\big(\omega^0\wedge\omega^1\big)=0.
\]
\end{remark}

\begin{proposition}\label{prop: exactness of L and tilde L}
   Let $\Omega\subset\mathbb{H}^n$ be a domain. For every integer $j\leqslant k$, the identity $\mathcal{L}_j^{(k)}\widetilde{\mathcal{L}}_j^{(k)}=0$ holds. Furthermore, suppose $\Omega\cap\left\{q_0=0\right\}=\emptyset$ and $\mathcal{L}_j^{(k)} f=0$ for all $0\leqslant j \leqslant k$. Then there exists a function $g\in C^\infty\big(\Omega,V_j^{(k+2)}\big)$ satisfying $\widetilde{\mathcal{L}}_j^{(k)}g=f$. For the special case $j=k+1$, the operator $\widetilde{\mathcal{L}}_{k+1}^{(k)}$ is surjective.
\end{proposition}

\begin{proof}
First, consider the case where $0\leqslant j \leqslant k$.
Recall the definitions of $\widetilde{\mathcal L}_j^{(k)}$ and $\mathcal{L}_j^{(k)}$ (see Definition \ref{def: L and tilde L}).
Direct computation yields
\begin{align*}
\eta\big(\mathcal{L}_j^{(k)}\widetilde{\mathcal L}_j^{(k)}h\big)_s
&= z_0^{0^\prime}\eta(f)_{s}+z_0^{1^\prime}\eta(f)_{s+1}
+ z_1^{0^\prime}\eta(g)_{s}+z_1^{1^\prime}\eta(g)_{s+1},
\end{align*}
where the equalities
\begin{align*}
\eta(f)_{s}=-z_1^{0'} \eta(h)_{s}-z_1^{1'} \eta(h)_{s+1},\qquad
\eta(g)_{s}=z_0^{0'} \eta(h)_{s}+z_0^{1'} \eta(h)_{s+1}
\end{align*}\
holds for any $s=0,1,\ldots, k-j+1$.
Let $H:=\eta(h)$. We obtain
\begin{align*}
\eta\big(\mathcal{L}_j^{(k)}\widetilde{\mathcal L}_j^{(k)}h\big)_s
&= -z_0^{0^\prime}\big(z_1^{0'} H_{s}+z_1^{1'} H_{s+1}\big)
-z_0^{1^\prime}\big(z_1^{0'} H_{s+1}+z_1^{1'} H_{s+2}\big)\\
&\quad + z_1^{0^\prime}\big(z_0^{0'} H_{s}+z_0^{1'} H_{s+1}\big)
+z_1^{1^\prime}\big(z_0^{0'} H_{s+1}+z_0^{1'} H_{s+2}\big)=0.
\end{align*}

Assume that $\mathcal{L}_j^{(k)}\bigl(f,g\bigr)^T = 0$. It suffices to construct a function $h$ satisfying
\[
\bigl(f,g\bigr)^T = \widetilde{\mathcal{L}}_j^{(k)} h.
\]
Then $h$ must satisfy
\begin{equation*}
\begin{pmatrix}
-z_1^{0'} & -z_1^{1'}\\
z_0^{0'} & z_0^{1'}
\end{pmatrix}
\begin{pmatrix}
\eta(h)_s\\
\eta(h)_{s+1}
\end{pmatrix}=
\begin{pmatrix}
\eta(f)_s\\
\eta(g)_s
\end{pmatrix},
\end{equation*}
where $s=0,1,\dots,k-j+1$. Consequently, we define $\eta(h)$ by
\begin{equation*}
\begin{pmatrix}
\eta(h)_s\\
\eta(h)_{s+1}
\end{pmatrix}\coloneqq \frac{1}{\vert q_0\vert^2}
\begin{pmatrix}
z_0^{1'} & z_1^{1'}\\
-z_0^{0'} & -z_1^{0'}
\end{pmatrix}
\begin{pmatrix}
\eta(f)_s\\
\eta(g)_s
\end{pmatrix}.
\end{equation*}
Note that for each $s=1,\dots,k-j+1$, we have
\[
\eta(h)_s=\frac{1}{\vert q_0\vert^2}\big(z_0^{1'}\eta(f)_{s}+z_1^{1'}\eta(g)_s\big)
\]
and
\[
\eta(h)_s=\frac{1}{\vert q_0\vert^2}\big(-z_0^{0'}\eta(f)_{s-1}-z_1^{0'}\eta(g)_{s-1}\big).
\]
It remains to verify that
\[
z_0^{1'}\eta(f)_{s}+z_1^{1'}\eta(g)_s+z_0^{0'}\eta(f)_{s-1}+z_1^{0'}\eta(g)_{s-1}=0.
\]
This identity follows directly from the condition $\mathcal{L}_j^{(k)}\bigl(f,g\bigr)^T = 0$ together with equation \eqref{eq:def mathcal L eta}.

Consider the case $j=k+1$. In this case, the operator
\[
\widetilde{\mathcal{L}}_{k+1}^{(k)}:C^\infty\big(\Omega,V_{k+1}^{(k+2)}\big)\longrightarrow C^\infty\big(\Omega,V_{k+1}^{(k+1)}\oplus V_{k+1}^{(k+1)}\big)
\]
is defined by
\[
\widetilde{\mathcal{L}}_{k+1}^{(k)}h=
\begin{pmatrix}
-z_1^{0'} & -z_1^{1'}\\
z_0^{0'} & z_0^{1'}
\end{pmatrix}
\begin{pmatrix}
h_{0'}\\
h_{1'}
\end{pmatrix}.
\]
Hence, the surjectivity of $\widetilde{\mathcal L}_{k+1}^{(k)}$ follows from the argument for the case $j\leqslant k$. This completes the proof.
\end{proof}

\begin{proposition}\label{prop: surjective of L}
      Let $\Omega\subset\mathbb{H}^n$ be a domain such that $\Omega\cap\{q_0=0\}=\emptyset$. For each $0\leqslant j\leqslant k$, the operator $\mathcal{L}_j^{(k)}$ is surjective.
\end{proposition}
\begin{proof}
   Recall the definitions of $\mathcal{L}_j^{(k)}$ (see Definition \ref{def: L and tilde L}) and equation \eqref{eq:def mathcal L eta}. We have
\begin{equation*}
\begin{aligned}
\begin{pmatrix}
\eta\big(\mathcal{L}_j^{(k)}\bigl(f,g\bigr)^T\big)_0 \\
\vdots \\
\eta\big(\mathcal{L}_j^{(k)}\bigl(f,g\bigr)^T\big)_{k-j}
\end{pmatrix}
&=
\begin{pmatrix}
z_0^{0'} & z_0^{1'}  &  & \\
& z_0^{0'} & z_0^{1'}  & \\
& & \ddots & \ddots
\end{pmatrix}
\begin{pmatrix}
\eta(f)_0 \\
\vdots \\
\eta(f)_{k-j+1}
\end{pmatrix}\\
&\quad +
\begin{pmatrix}
z_1^{0'} & z_1^{1'}  &  & \\
& z_1^{0'} & z_1^{1'}  & \\
& & \ddots & \ddots
\end{pmatrix}
\begin{pmatrix}
\eta(g)_0 \\
\vdots \\
\eta(g)_{k-j+1}
\end{pmatrix}.
\end{aligned}
\end{equation*}
From the definition of $z_A^{A'}$ given in \eqref{eq:def:zAAprime}, it follows that
\begin{equation*}
\begin{aligned}
&z_0^{0'}\overline{z_0^{0'}}+z_0^{1'}\overline{z_0^{1'}}+z_1^{0'}\overline{z_1^{0'}}+z_1^{1'}\overline{z_1^{1'}}=2|q_0|^2,\\
&z_0^{1'}\overline{z_0^{0'}}+z_1^{1'}\overline{z_1^{0'}}=0. \\
\end{aligned}
\end{equation*}
  This implies that
\begin{equation*}
\begin{aligned}
2|q_0|^2\mathrm{Id}_{k-j+1}
&=
\begin{pmatrix}
z_0^{0'} & z_0^{1'}  &  & \\
& z_0^{0'} & z_0^{1'}  & \\
& & \ddots & \ddots
\end{pmatrix}
\begin{pmatrix}
\overline{z_0^{0'}} & \overline{z_0^{1'}}  &  & \\
& \overline{z_0^{0'}} & \overline{z_0^{1'}}  & \\
& & \ddots & \ddots
\end{pmatrix}^{\!T} \\
&\quad +
\begin{pmatrix}
z_1^{0'} & z_1^{1'}  &  & \\
& z_1^{0'} & z_1^{1'}  & \\
& & \ddots & \ddots
\end{pmatrix}
\begin{pmatrix}
\overline{z_1^{0'}} & \overline{z_1^{1'}}  &  & \\
& \overline{z_1^{0'}} & \overline{z_1^{1'}}  & \\
& & \ddots & \ddots
\end{pmatrix}^{\!T}.
\end{aligned}
\end{equation*}

For each $h\in C^\infty(\Omega,V_j^{(k)})$, we may choose $f$ and $g$ such that
\begin{align*}
&\eta(f)=\frac{1}{2|q_0|^2}\begin{pmatrix}
\overline{z_0^{0'}} & \overline{z_0^{1'}}  &  & \\
& \overline{z_0^{0'}} & \overline{z_0^{1'}}  & \\
& & \ddots & \ddots
\end{pmatrix}^{\!T} \eta(h),\\
&\eta(g)=\frac{1}{2|q_0|^2}\begin{pmatrix}
\overline{z_1^{0'}} & \overline{z_1^{1'}}  &  & \\
& \overline{z_1^{0'}} & \overline{z_1^{1'}}  & \\
& & \ddots & \ddots
\end{pmatrix}^{\!T}\eta(h).
\end{align*}
From the above arguments, we conclude that $\mathcal{L}_j^{(k)}\bigl(f,g\bigr)^T=h$.
\end{proof}

As is well known in homological algebra \cite{Rotman09}, Proposition \ref{prop:DL and D tilde L commutative} induces two maps
\[
\mathcal{L}_j^{(k)}:H^j\big(\Omega,\mathcal{R}^{(k+1)}\oplus \mathcal{R}^{(k+1)}\big)\longrightarrow H^j\big(\Omega,\mathcal{R}^{(k)}\big)
\]
and
\[
\widetilde{\mathcal{L}}_j^{(k)}:H^j\big(\Omega,\mathcal{R}^{(k+2)}\big)\longrightarrow H^j\big(\Omega,\mathcal{R}^{(k+1)}\oplus \mathcal{R}^{(k+1)}\big).
\]
Likewise, the connecting homomorphisms can also be constructed for $j\leqslant k$.

\begin{proposition}\label{prop: connecting homomorphism s}
Let $\Omega\subset \mathbb{H}^n$ be a domain satisfying $\Omega\cap X=\emptyset$. For each integer $j\leqslant k$, there exists a well-defined homomorphism
\[
s^{(k)}_j:H^j\big(\Omega,\mathcal{R}^{(k)}\big)\longrightarrow H^{j+1}\big(\Omega,\mathcal{R}^{(k+2)}\big)
\]
given by
\begin{equation}\label{eq:def:snake map}
\operatorname{cls}\big(s_j^{(k)}f\big) \coloneqq \operatorname{cls}\Big(\big(\widetilde{\mathcal{L}}_{j+1}^{(k)}\big)^{-1}\Big(\mathscr{D}_j^{(k+1)}\big(\mathcal{L}_j^{(k)}\big)^{-1}f\Big)\Big).
\end{equation}
\end{proposition}
\begin{proof}
    It follows from the proof of Proposition 6.9 in \cite{Rotman09}, Proposition \ref{prop: injective of tilde L}, Proposition \ref{prop: exactness of L and tilde L}, and Proposition \ref{prop: surjective of L} that $s^{(k)}_j$ is well-defined for each integer $0\leqslant j\leqslant k-1$. For the case $j = k$, we observe that the operator $\widetilde{\mathcal L}_{k+2}^{(k)}$ is not injective. When attempting to prove that $\mathscr{D}^{(k+2)}_{k+1}s_k^{(k)}f = 0$ under the condition $\mathscr{D}_k^{(k)}f = 0$, we run into obstacles if we adopt the argument strategy used to establish Proposition 6.9 from \cite{Rotman09}.

    We verify via direct computation that $\mathscr{D}_k^{(k)}f = 0$ implies $\mathscr{D}^{(k+2)}_{k+1}s_k^{(k)}f = 0$. The argument follows the same line of reasoning as the proof of Theorem 1.2 in \cite{LZ26}. Note that $V_k^{(k)}\cong \Lambda^k\mathbb{C}^{2n}$ and $V_{k+1}^{(k+2)}\cong \mathbb{C}^2\otimes\Lambda^{k+1}\mathbb{C}^{2n}$. Let $f\in C^\infty(\Omega,\Lambda^k\mathbb{C}^{2n})$. Then the operator $s_k^{(k)}f$ is defined as follows:
\begin{align*}
2\left[\begin{pmatrix}
    (s_k^{(k)}f)_{0'}\\
    (s_k^{(k)}f)_{1'}
\end{pmatrix}\right]&=\left[\frac{1}{|q_0|^2}\begin{pmatrix}
z_0^{1'} & z_1^{1'}\\
-z_0^{0'} & -z_1^{0'}
\end{pmatrix}
\begin{pmatrix}
d^{0'}\!\left(\frac{\overline{z_0^{0'}}f}{|q_0|^2}\right)+d^{1'}\!\left(\frac{\overline{z_0^{1'}}f}{|q_0|^2}\right)\\
d^{0'}\!\left(\frac{\overline{z_1^{0'}}f}{|q_0|^2}\right)+d^{1'}\!\left(\frac{\overline{z_1^{1'}}f}{|q_0|^2}\right)
\end{pmatrix}\right].
\end{align*}
It follows from the definition of $z_A^{A'}$ given in \eqref{eq:z AAprime} and Lemma \ref{lem:computation z} that
\begin{equation*}
\begin{aligned}
d^{0'}\!\left(\frac{\overline{z_0^{0'}}f}{|q_0|^2}\right)+d^{1'}\!\left(\frac{\overline{z_0^{1'}}f}{|q_0|^2}\right)
&=\frac{\overline{z_0^{0'}}}{|q_0|^2}d^{0'}f+\frac{\overline{z_0^{1'}}}{|q_0|^2}d^{1'}f\\
&\quad -\frac{2}{|q_0|^2}\omega^0\wedge f,\\
d^{0'}\!\left(\frac{\overline{z_1^{0'}}f}{|q_0|^2}\right)+d^{1'}\!\left(\frac{\overline{z_1^{1'}}f}{|q_0|^2}\right)
&=\frac{\overline{z_1^{0'}}}{|q_0|^2}d^{0'}f+\frac{\overline{z_1^{1'}}}{|q_0|^2}d^{1'}f\\
&\quad -\frac{2}{|q_0|^2}\omega^1\wedge f.
\end{aligned}
\end{equation*}
  This implies that
\begin{align*}
&\quad  \frac{1}{|q_0|^2}
\begin{pmatrix}
z_0^{1'} & z_1^{1'}\\
-z_0^{0'} & -z_1^{0'}
\end{pmatrix}
\begin{pmatrix}
d^{0'}\!\left(\dfrac{\overline{z_0^{0'}}f}{|q_0|^2}\right)
+d^{1'}\!\left(\dfrac{\overline{z_0^{1'}}f}{|q_0|^2}\right)\\
d^{0'}\!\left(\dfrac{\overline{z_1^{0'}}f}{|q_0|^2}\right)
+d^{1'}\!\left(\dfrac{\overline{z_1^{1'}}f}{|q_0|^2}\right)
\end{pmatrix}\\
&=
\begin{pmatrix}
\frac{1}{|q_0|^2}d^{1'}f
-\frac{2z_0^{1'}}{|q_0|^4}\omega^0\wedge f
-\frac{2z_1^{1'}}{|q_0|^4}\omega^1\wedge f \\
-\frac{1}{|q_0|^2}d^{0'}f
+\frac{2z_0^{0'}}{|q_0|^4}\omega^0\wedge f
+\frac{2z_1^{0'}}{|q_0|^4}\omega^1\wedge f
\end{pmatrix}\\
&=
\begin{pmatrix}
\frac{1}{|q_0|^2}d^{1'}f
-\left(d^{1'}\frac{1}{|q_0|^2}\right)\wedge f \\
- \frac{1}{|q_0|^2}d^{0'}f
+\left(d^{0'}\frac{1}{|q_0|^2}\right)\wedge f
\end{pmatrix}.
\end{align*}
Assume that $d^{0'}d^{1'}f=0$. We aim to prove that $\mathscr{D}_{k+1}^{(k+2)}\big(s_k^{(k)}f\big)=0$. It is easy to verify that $d^{0'}d^{1'}\dfrac{1}{|q_0|^2}=0$. By direct computation, we have
\begin{align*}
\mathscr{D}_{k+1}^{(k+2)}\big(2s_k^{(k)}f\big)
&= d^{0'}\!\left(\frac{1}{|q_0|^2}d^{1'}f
-\left(d^{1'}\frac{1}{|q_0|^2}\right)\wedge f\right)
+ d^{1'}\!\left(- \frac{1}{|q_0|^2}d^{0'}f
+\left(d^{0'}\frac{1}{|q_0|^2}\right)\wedge f\right)\\
&= \frac{2}{|q_0|^2}d^{0'}d^{1'}f = 0.
\end{align*}

On the other hand, suppose $f=d^{0'}F_{0'}+d^{1'}F_{1'}$. Then
\begin{align*}
2\begin{pmatrix}
\big(s_k^{(k)}f\big)_{0'}\\
\big(s_k^{(k)}f\big)_{1'}
\end{pmatrix}&=\begin{pmatrix}
\dfrac{1}{|q_0|^2}d^{1'}d^{0'}F_{0'}
-\left(d^{1'}\dfrac{1}{|q_0|^2}\right)\wedge \big(d^{0'}F_{0'}+d^{1'}F_{1'}\big) \\
- \dfrac{1}{|q_0|^2}d^{0'}d^{1'}F_{1'}
+\left(d^{0'}\dfrac{1}{|q_0|^2}\right)\wedge \big(d^{0'}F_{0'}+d^{1'}F_{1'}\big)
\end{pmatrix}.
\end{align*}
Define $G\in C^{\infty}(\Omega,V_{k}^{(k+2)})$ by
\begin{align*}
G_{0'0'}&=2d^{1'}\frac{1}{|q_0|^2}\wedge F_{0'},\\
G_{0'1'}&=G_{1'0'}=\frac{1}{|q_0|^2}\big(d^{0'}F_{0'}-d^{1'}F_{1'}\big),\\
G_{1'1'}&=-2d^{0'}\frac{1}{|q_0|^2}\wedge F_{1'}.
\end{align*}
It then follows that $\mathscr{D}_{k}^{(k+2)}G=2s_k^{(k)}f$, which implies that the operator $s_k^{(k)}$ is well-defined. This completes the proof.
\end{proof}

\emph{Proof of Theorem \ref{thm:long exact sequence}.}
The proof of Theorem \ref{thm:long exact sequence} follows from Proposition \ref{prop:DL and D tilde L commutative}, Proposition \ref{prop: injective of tilde L}, Proposition \ref{prop: exactness of L and tilde L}, Proposition \ref{prop: surjective of L}, Proposition \ref{prop: connecting homomorphism s} and proof of  Theorem 6.10 in \cite{Rotman09}. \qed

\section{Applications}\label{section:Application}

In this section, we discuss several applications of our main theoretical results, including the surjectivity of the operator $\mathcal{L}_0^{(1)}$ in the case $n=1$ and conditions for the vanishing of cohomology groups.

\subsection{The surjectivity of operator $(L_0^{(1)},L_1^{(1)})$ in the case $n=1$}

In this section, we give a proof of Theorem \ref{thm:2}.

\emph{Proof of Theorem \ref{thm:2}.}
By the definition of ideal sheaf $\mathcal{I}_{\{0\}}^{(0)}$, we only need to prove that the map
\[(L_0^{(1)},L_1^{(1)}):H^0\big(\Omega,\mathcal{R}^{(1)}\oplus \mathcal{R}^{(1)}\big)\longrightarrow H^0(\Omega,\mathcal{I}_{\{0\}}^{(0)})\]
is surjective.

  As established in \cite{LZ26}, the identity $H^1(\Omega,\mathcal{R}^{(2)})=0$ holds if and only if $H^3(\Omega,\mathbb{R})=0$.
Therefore, it suffices to prove that $H^1(\Omega,\mathcal{R}^{(2)})=0$ is equivalent to the surjectivity of the map $(L_0^{(1)},L_1^{(1)})$.

From the vanishing result $H^1(\Omega,\mathcal{R}^{(1)})=0$ and Theorem \ref{thm:main}, we deduce the existence of a long exact sequence for cohomology groups:
\[
H^0\big(\Omega,\mathcal{R}^{(1)}\oplus \mathcal{R}^{(1)}\big)\stackrel{(L_0^{(1)},L_1^{(1)})}{\longrightarrow} H^0(\Omega,\mathcal{I}_{\{0\}}^{(0)}) \longrightarrow H^1(\Omega,\mathcal{R}^{(2)}) \longrightarrow 0.
\]
Consequently, the map $(L_0^{(1)},L_1^{(1)})$ is surjective if and only if $H^1(\Omega,\mathcal{R}^{(2)})=0$. This completes the proof. \qed

\subsection{Vanishing of cohomology groups and extension of $k$-regular functions}

In this section, we derive several sufficient conditions for extending $k$-regular functions defined on the subset $X\subset\mathbb{H}^{n-1}$ to the entire domain $\Omega\subset\mathbb{R}^{4n}$. As stated in Proposition \ref{prop:extension}, our analysis relies on results concerning the vanishing property of the first cohomology group $H^1(\Omega,\mathcal{I}^{(k)}_X)$.

Theorem \ref{thm:main} provides a sheaf-theoretic resolution of $\mathcal{I}_X^{(k)}$. To connect it with the differential operators $\mathscr{D}_j^{(k)}$ and the long exact sequence of Theorem \ref{thm:long exact sequence}, we need to relate sheaf cohomology to \v{C}ech cohomology.

We will recall some knowledge of \v{C}ech cohomology. The notation we adopt here refers to \cite{Demailly12}.  Let $\Omega\subset \mathbb{R}^{4n}$ be a domain, set $X=\{q_0=0\}\cap \Omega$, and let $\{U_\alpha\}$ form a convex open cover of $\Omega$. Let $\mathscr{A}$ be a sheaf over domain $\Omega$.
The group $C^{p}(\mathscr{U},\mathscr{A})$ of \v{C}ech $p$-cochains is defined as the collection of all families $c=(c_{\alpha_0\cdots\alpha_p})$ with
\[c_{\alpha_0\cdots \alpha_p}\in \prod_{\alpha_0,\dots,\alpha_p} \mathscr{A}(U_{\alpha_0\cdots \alpha_p}).\]
The \v{C}ech differential $\delta^p:C^{p}(\mathscr{U},\mathscr{A})\to C^{p+1}(\mathscr{U},\mathscr{A})$ is given by
\[(\delta^pc)_{\alpha_0\cdots \alpha_{p+1}}(x):=\sum_{0\leqslant j\leqslant p+1}(-1)^jc_{\alpha_0\cdots \widehat{\alpha_j}\cdots \alpha_{p+1}}.\]
We define the group $\check{H}^p(\mathscr{U},\mathscr{A})$ as
\[\check{H}^p(\mathscr{U},\mathscr{A}):=\frac{\ker \delta^p}{\operatorname{Im}\delta^{p-1}}.\]

\begin{lemma}\label{lem:Cech coho and usual coho}
  Let $\Omega\subset \mathbb{R}^{4n}$ be a domain, set $X=\{q_0=0\}\cap \Omega$, and let $\{U_\alpha\}$ form a convex open cover of $\Omega$. Then we have the isomorphisms
\[
\check{H}^p(\mathscr{U},\mathcal{R}^{(k)})\cong H^p(\Omega,\mathcal{R}^{(k)}), \qquad \check{H}^p(\mathscr{U},\mathcal{I}_X^{(k)})\cong H^p(\Omega,\mathcal{I}_X^{(k)})
\]
for $p\geqslant 0$.
\end{lemma}
\begin{proof}
The isomorphism is valid in the case $p=0$, since we have
\[
\check{H}^0(\mathscr{U},\mathcal{R}^{(k)})\cong H^0(\Omega,\mathcal{R}^{(k)})\cong \mathcal{R}^{(k)}(\Omega), \qquad \check{H}^0(\mathscr{U},\mathcal{I}_X^{(k)})\cong H^0(\Omega,\mathcal{I}_X^{(k)})\cong\mathcal{I}_X^{(k)}(\Omega).
\]
We now turn to the case where $p\geqslant 1$. Recall from \cite{Wang10} that for any convex open set $U$, we have $H^p(U,\mathcal{R}^{(k)}) = 0$ for $p\geqslant 1$. The isomorphism $\check{H}^p(\mathscr{U},\mathcal{R}^{(k)}) \cong H^p(\Omega,\mathcal{R}^{(k)})$ follows directly from Leray's isomorphism theorem, see Theorem 5.17 in \cite{Demailly12}.

To establish $\check{H}^p(\mathscr{U},\mathcal{I}_X^{(k)}) \cong H^p(\Omega,\mathcal{I}_X^{(k)})$, it suffices by Leray's isomorphism to verify $H^p(U,\mathcal{I}_X^{(k)}) = 0$ for all convex open sets $U$.

According to Theorem \ref{thm:main}, we have the following exact sequence:
\[H^p(U,\mathcal{R}^{(k+1)}\oplus\mathcal{R}^{(k+1)})\longrightarrow H^p(U,\mathcal{I}_X^{(k)})\longrightarrow H^{p+1}(U,\mathcal{R}^{(k+2)}).\]
As stated in \cite{Wang10}, for any $p\geqslant1$, we have
\[H^p(U,\mathcal{R}^{(k+1)}\oplus\mathcal{R}^{(k+1)})=0,\qquad H^{p+1}(U,\mathcal{R}^{(k+2)})=0.\]
Combining the above facts, we conclude $H^p(U,\mathcal{I}_X^{(k)})=0$.
\end{proof}

We now introduce the Leray isomorphism
\begin{equation*}
\upsilon^{p,(k)}:H^p(\Omega,\mathcal{R}^{(k)})\longrightarrow \check{H}^p(\mathscr{U},\mathcal{R}^{(k)}).
\end{equation*}
For details concerning the Leray isomorphism, we refer to Theorem 5.17 in \cite{Demailly12}.

Consider the groups $K^{p,q}(\mathscr{U},V_q^{(k)})$ defined by
\[K^{p,q}(\mathscr{U},V_q^{(k)}):=\prod_{\alpha_0,\dots,\alpha_p} C^\infty(U_{\alpha_0\cdots \alpha_p},V_q^{(k)}).\]
We have two differentials $\delta^p\colon K^{p,q}\to K^{p+1,q}$ and $\mathscr{D}_q^{(k)}\colon K^{p,q}\to K^{p,q+1}$.

Note that any finite intersection $U_{\alpha\beta\cdots}:=U_\alpha\cap U_\beta\cap\cdots$ is also an open convex set, and the cohomology groups $H^q(U_{\alpha\beta\cdots},\mathcal{R}^{(k)})$ vanish for all $q\geqslant 1$ \cite{Wang10}. Hence, for any function $f$ satisfying $\mathscr{D}_{q+1}^{(k)}f=0$ on $U_{\alpha_0\cdots\alpha_p}$, there exists a function $F$ defined on $U_{\alpha_0\cdots\alpha_p}$ such that $\mathscr{D}_q^{(k)}F=f$. We denote such $F$ as $(\mathscr{D}_q^{(k)})^{-1}f$.

Let $[f]\in H^{p}(\Omega,\mathcal{R}^{(k)})$ and take $f$ to be a representative of the equivalence class $[f]$. Then we may define a section $F\in \prod_{\alpha}C^\infty(U_\alpha,V_p^{(k)})$ satisfying
\[F_\alpha=f\big|_{U_\alpha}.\]
The operator $\upsilon^{p,(k)}$ is defined by
\begin{equation}\label{eq:def:upsilon pk}
(\upsilon^{p,(k)}f)_{\alpha_0\cdots\alpha_p}:=(\delta^{p}\circ(\mathscr{D}_{0}^{(k)})^{-1})\circ\cdots\circ(\delta^2\circ(\mathscr{D}_{p-2}^{(k)})^{-1})\circ(\delta^1\circ(\mathscr{D}_{p-1}^{(k)})^{-1})F,
\end{equation}
where the inverses are well-defined by the vanishing of higher cohomology on convex sets.

We explicitly present the expressions of $\upsilon^{p,(k)}$ for the special cases where $p=1,2$.

Let $f\in C^{\infty}(\Omega,V_1^{(k)})$ satisfy the equation $\mathscr{D}_1^{(k)}f=0$. Then on each open set $U_\alpha$, there exists a function $f_\alpha\in C^\infty(U_\alpha,V_0^{(k)})$ such that $\mathscr{D}_0^{(k)}f_\alpha=f$. The element $\upsilon^{1,(k)}f$ is then defined as a section on $\prod_{\alpha\beta}\mathcal{R}^{(k)}(U_{\alpha\beta})$ via
\begin{equation*}
(\upsilon^{1,(k)}f)_{\alpha\beta}:=f_\beta-f_\alpha.
\end{equation*}

Let $f\in C^{\infty}(\Omega,V_2^{(k)})$ satisfy the equation $\mathscr{D}_2^{(k)}f=0$. Then for each open set $U_\alpha$, there exists a function $f_\alpha\in C^\infty(U_\alpha,V_1^{(k)})$ such that $\mathscr{D}_1^{(k)}f_\alpha=f$. The expression $\delta^1(f_\alpha)$ is then regarded as an element of $\prod_{\alpha\beta}C^\infty(U_{\alpha\beta},V_1^{(k)})$ defined by
\[
(\delta^1(f_\alpha))_{{\alpha\beta}}:=f_\beta-f_\alpha.
\]
Notice that $\mathscr{D}_1^{(k)}(f_\beta-f_\alpha)=0$. Hence there exists $F_{\alpha\beta}$ satisfying $\mathscr{D}_0^{(k)}F_{\alpha\beta}=f_\beta-f_\alpha$.
The element $\upsilon^{2,(k)}f$ is subsequently defined as a section over $\prod_{\alpha\beta\gamma}\mathcal{R}^{(k)}(U_{\alpha\beta\gamma})$ via
\begin{equation*}
\big(\upsilon^{2,(k)}f\big)_{\alpha\beta\gamma}:=F_{\beta\gamma}-F_{\alpha\gamma}+F_{\alpha\beta}.
\end{equation*}

Note that Proposition \ref{prop:DL and D tilde L commutative} gives rise to a map
\[
\widetilde{\mathscr{L}}^{(k)}_p:\prod_{\alpha_0,\ldots,\alpha_p}C^\infty(U_{\alpha_0\ldots\alpha_p},\mathcal{R}^{(k+2)})\longrightarrow \prod_{\alpha_0,\ldots,\alpha_p}C^\infty(U_{\alpha_0\ldots\alpha_p},\,\mathcal{R}^{(k+1)}\oplus\mathcal{R}^{(k+1)})
\]
defined by
\[
\widetilde{\mathscr{L}}^{(k)}_p f_{\alpha_0\cdots\alpha_p}
:=(\widetilde{\mathcal L}_0^{(k)}f)_{\alpha_0\cdots\alpha_p}=
\bigl(-L_1 f_{\alpha_0\cdots\alpha_p},\, L_0 f_{\alpha_0\cdots\alpha_p}\bigr)^T.
\]
We then have the following lemma concerning the connection between the map
\[
\widetilde{\mathscr{L}}^{(k)}_p:\check{H}^{p}(\mathscr{U},\mathcal{R}^{(k+2)})\longrightarrow \check{H}^{p}\big(\mathscr{U},\mathcal{R}^{(k+1)}\oplus \mathcal{R}^{(k+1)}\big)
\]
and
\[
\widetilde{\mathcal{L}}^{(k)}_p:H^{p}(\Omega,\mathcal{R}^{(k+2)})\longrightarrow H^{p}(\Omega,\mathcal{R}^{(k+1)}\oplus \mathcal{R}^{(k+1)}).
\]

\begin{lemma}\label{lem:connection between tilde L and msthscr L}
    Let $n\geqslant 1$, $k,p\geqslant 0$ and $\Omega\subset\mathbb{R}^{4n}$ be a domain. Then we have
    \begin{enumerate}
        \item We have
        \[\upsilon^{p,(k+1)}\widetilde{\mathcal L}^{(k)}_p=\widetilde{\mathscr L}^{(k)}_p\upsilon^{p,(k+2)}.\]
        \item If $\widetilde{\mathcal{L}}^{(k)}_p$ is surjective, then $\widetilde{\mathscr{L}}^{(k)}_p$ is surjective.
        \item If $\widetilde{\mathcal{L}}^{(k)}_p$ is injective, then $\widetilde{\mathscr{L}}^{(k)}_p$ is injective.
    \end{enumerate}
\end{lemma}
\begin{proof}
   We first prove assertion (1), since (2) and (3) follow as corollaries of (1).

From the definition of $\widetilde{\mathcal L}_p^{(k)}$ (see Definition \ref{def: L and tilde L}) and Proposition \ref{prop:DL and D tilde L commutative}, we know the following identities hold for all
\[f\in \prod_{\alpha_0,\ldots,\alpha_q}C^{\infty}\big(U_{\alpha_0\ldots\alpha_q},V_{p-1}^{(k+2)}\big):\]
\[
\mathscr{D}_{p-1}^{(k+1)}\widetilde{\mathcal L}_{p-1}^{(k)}f=\widetilde{\mathcal L}_{p}^{(k)}\mathscr{D}_{p-1}^{(k+2)}f, \qquad \delta^q\widetilde{\mathcal L}_{p-1}^{(k)}f=\widetilde{\mathcal L}_{p-1}^{(k)}\delta^q f.
\]

Notice that $\widetilde{\mathcal L}_p^{(k)}$, $\widetilde{\mathscr L}_p^{(k)}$ and $\upsilon^{p,(k)}$ are morphisms between cohomology groups, and their formulas are independent of the choice of representatives of equivalence classes. Thus we may select a representative of an equivalence class for computation.
Let $f\in H^p(\Omega,\mathcal{R}^{(k+2)})$ and $F\in\prod_{\alpha}C^\infty(U_\alpha,V_p^{(k+2)})$ satisfy $F_\alpha=f\big|_{U_\alpha}$. Then we obtain
\begin{align*}
(\delta^1\circ(\mathscr{D}_{p-1}^{(k+1)})^{-1})\widetilde{\mathcal L}^{(k)}_pF&=
(\delta^1\circ(\mathscr{D}_{p-1}^{(k+1)})^{-1})\widetilde{\mathcal L}^{(k)}_p\mathscr{D}_{p-1}^{(k+2)}(\mathscr{D}_{p-1}^{(k+2)})^{-1}F\\
&=\widetilde{\mathcal L}^{(k)}_{p-1}\circ\big(\delta^1\circ (\mathscr{D}_{p-1}^{(k+2)})^{-1}\big)F.
\end{align*}
Repeating the above argument yields
\begin{align*}
\big(\upsilon^{p,(k+1)}\widetilde{\mathcal L}_p^{(k)}f\big)_{\alpha_0\cdots\alpha_p}&=\widetilde{\mathcal L}^{(k)}_{0}\circ \big(\delta^p\circ (\mathscr{D}_{0}^{(k+2)})^{-1}\big)F\circ \cdots\circ \big(\delta^1\circ (\mathscr{D}_{p-1}^{(k+2)})^{-1}\big)F\\
&=\big(\widetilde{\mathscr{L}}^{(k)}_p\upsilon^{p,(k+2)}f\big)_{\alpha_0\cdots\alpha_p}.
\end{align*}
This finishes the proof.
    \end{proof}

Now we present a proof for Theorem \ref{thm:main2}.

\emph{Proof of Theorem \ref{thm:main2}.}

The assertion (2) is a direct corollary of Theorem \ref{thm:long exact sequence}. It suffices to prove assertion (1).

Let $\Omega\subset \mathbb{R}^{4n}$ be a domain, and define $X=\{q_0=0\}\cap \Omega$. Suppose $\{U_\alpha\}$ is a convex open cover of $\Omega$.

It follows from Lemma \ref{lem:Cech coho and usual coho} that
\[H^1(\Omega,\mathcal{I}_X^{(k)})\cong \check{H}^1(\mathscr{U},\mathcal{I}_X^{(k)}).\]
By Theorem \ref{thm:main}, we obtain the following long exact sequence of cohomology groups:
\begin{equation*}
\begin{aligned}
&\check{H}^{1}\big(\mathscr{U},\mathcal{R}^{(k+2)}\big)\xrightarrow{\widetilde{\mathscr{L}}^{(k)}_1} \check{H}^{1}\big(\mathscr{U},\mathcal{R}^{(k+1)}\oplus\mathcal{R}^{(k+1)}\big)\longrightarrow \check{H}^1\big(\mathscr{U},\mathcal{I}_X^{(k)}\big)\longrightarrow \\
&\check{H}^{2}\big(\mathscr{U},\mathcal{R}^{(k+2)}\big)\xrightarrow{\widetilde{\mathscr{L}}^{(k)}_2} \check{H}^{2}\big(\mathscr{U},\mathcal{R}^{(k+1)}\oplus\mathcal{R}^{(k+1)}\big).
\end{aligned}
\end{equation*}
Consequently, the cohomology group
\[
\check{H}^1\big(\mathscr{U},\mathcal{I}_X^{(k)}\big)
\]
vanishes if and only if the map $\widetilde{\mathscr{L}}^{(k)}_1$ is surjective and $\widetilde{\mathscr{L}}^{(k)}_2$ is injective. Combining this with Lemma \ref{lem:connection between tilde L and msthscr L}, we arrive at the desired conclusion. \qed

 Corollaries \ref{cor:1} and \ref{cor:2} follow directly from Theorem \ref{thm:main}, Theorem \ref{thm:main2} and Proposition \ref{prop:extension}.

\begin{remark}
   We must point out that Corollary \ref{cor:2} provides a sufficient but not necessary condition for extending $k$-regular functions from $X$ to $\Omega$. Let us consider the case where $k=n=1$. In this setting, every constant function can be trivially extended from the single point $0$ to the entire space $\Omega$. However, as noted in \cite{LZ26}, $H^1(\Omega,\mathcal{R}^{(2)})=0$ if and only if $H^3(\Omega,\mathbb{R})=0$.
\end{remark}

\end{document}